\documentclass[]{article}
\usepackage{graphicx}
\usepackage{amsmath,amssymb,amsthm,amsfonts,amsbsy}
\usepackage[english]{babel}
\usepackage[T1]{fontenc}
\usepackage[utf8]{inputenc}
\usepackage{graphicx}
\usepackage{hyperref}

\usepackage{comment}

\usepackage{pdfpages}

\usepackage{multirow}
\usepackage{booktabs}

\usepackage{tikz}
\usepackage{makeidx}
\usepackage{subcaption}
\usepackage{float}
\usepackage{matlab-prettifier}

\usepackage{xcolor}

\DeclareMathOperator{\vect}{vec}

\usepackage{hyperref}

\theoremstyle{plain}
\newtheorem{theorem}{Theorem}[section]
\newtheorem{lemma}{Lemma}[section]

  \newtheorem{assumption}{Assumption}[section]

\theoremstyle{definition}
\newtheorem{definition}{Definition}[section]

\DeclareMathOperator*{\argmin}{arg\,min}

\newcommand{\bw}{\mathbf{w}}

\newcommand{\bz}{\mathbf{z}}

\newcommand{\bzeros}{\mathbf{0}}

\newcommand{\be}{\mathbf{e}}

\newcommand{\bA}{\mathbf{A}}
\newcommand{\bB}{\mathbf{B}}
\newcommand{\bY}{\mathbf{Y}}
\newcommand{\bG}{\mathbf{G}}

\newcommand{\bE}{\mathbf{E}}
\newcommand{\bF}{\mathbf{F}}

\newcommand{\bS}{\mathbf{S}}
\newcommand{\bT}{\mathbf{T}}
\newcommand{\bZ}{\mathbf{Z}}
\newcommand{\bM}{\mathbf{M}}
\newcommand{\bN}{\mathbf{N}}
\newcommand{\bC}{\mathbf{C}}
\newcommand{\bQ}{\mathbf{Q}}
\newcommand{\bX}{\mathbf{X}}
\newcommand{\bI}{\mathbf{I}}
\newcommand{\bJ}{\mathbf{J}}
\newcommand{\bK}{\mathbf{K}}
\newcommand{\bH}{\mathbf{H}}
\newcommand{\bU}{\mathbf{U}}
\newcommand{\bV}{\mathbf{V}}
\newcommand{\bW}{\mathbf{W}}
\newcommand{\bP}{\mathbf{P}}

\newcommand{\bLambda}{\boldsymbol{\Lambda}}

\newcommand{\bGamma}{\mathbf{\Gamma}}
\newcommand{\bUpsilon}{\mathbf{\Upsilon}}
\newcommand{\bPhi}{\mathbf{\Phi}}
\newcommand{\bPsi}{\mathbf{\Psi}}
\newcommand{\bOmega}{\mathbf{\Omega}}

\newcommand{\bLagrange}{\boldsymbol{\cal L}}

\newcommand{\nul}{\mathcal{N}} 

\newcommand{\R}{{\mathbb R}}

\newcommand{\twopartdef}[4]

\newcommand{\PROJ}{\boldsymbol{\mathcal{P}}}

\title{Linear Algebra Problems Solved by Using Damped Dynamical Systems on the Stiefel Manifold\\

\author{ Mårten Gulliksson$^{(a)}$\thanks{marten.gulliksson@oru.se}, Anna Oleynik$^{(b)}$\thanks{anna.oleynik@uib.no}, Magnus Ögren$^{(a)}$\thanks{magnus.ogren@oru.se}, Rohollah Bakhshandeh-Chamazkoti$^{(c)}$\thanks{r\_bakhshandeh@nit.ac.ir}
\\
$^{(a)}$School of Science and Technology, Örebro University, 70182 Örebro, Sweden
\\
$^{(b)}$Department of Mathematics, Bergen University, Bergen, Norway
\\
$^{(c)}$Faculty of Basic Sciences, Babol Noshirvani University of Technology, Babol, Iran.
\\
}
}

\begin{document}

\maketitle
\section*{Abstract}
We develop a new method for solving minimization problems on the Stiefel Manifold using damped dynamical systems. The constraints are satisfied in the limit by an additional damped dynamical system. The method is illustrated by numerical experiments and compared to a state-of-the-art conjugate gradient method.\\
\textbf{Keywords}:
Optimization; Constraints; Second-order damped dynamical systems; Stiefel manifold.

\section{Introduction}

We describe a method for finding the minimum of a non-linear matrix function on the Stiefel manifold, i.e., the set of orthonormal matrices. The approach is based on the introduction of a second-order damped dynamical system generating a trajectory whose limit point is the solution of the minimization problem. 

We consider two different formulations of such dynamical systems: 1. Using the Lagrange function not requiring the trajectory to stay on the manifold. 2. Staying on the manifold with the use of projection.  When using the Lagrange function we define an additional damped system for the constraints that is then used to get the Lagrange parameters and consequently attaining a dynamical system only in the unknown matrix.

Throughout the paper, we use the linear eigenvalue and Procrustes problems as illustrations of the approach. 

\section{Optimizing on the Stiefel manifold}
\label{sec:appl}

Assume the function $F: \mathbb{R}^{n \times p} \longrightarrow \mathbb{R}$ to be twice continuously differentiable, and denote 
$$
\frac{\partial F}{\partial \bX}(\bX)=\bG(\bX).
$$
The minimization problem under consideration is
\begin{equation}\label{eqas}
\min_{\bX\in {\rm St}(n,p)}
F(\bX),
\end{equation}
where ${\rm St}(n,p)$ is the Stiefel manifold, 
$${\rm St}(n,p):= \{ \bX \in \R^{n \times p} ~|~ \bX^{\top} \bX = {\rm I}_p \}.$$
The Stiefel manifold can be embedded in the $np$-dimensional Euclidean space of $n\times p$ matrices. The tangent space of ${\rm St}(n,p)$ at a point $\bX$ can be defined as 
\begin{equation*}
\begin{aligned}
     T_\bX {\rm St}(n,p) &= \{ \bZ\in\R ^{n\times p} ~|~ \bX^{\top} \bZ + \bZ^{\top} \bX = 0 \}.
\end{aligned}
\end{equation*}
The normal space is defined to be the orthogonal complement of the tangent
space. Orthogonality depends upon the definition of an inner product, and because we view the Stiefel manifold as an embedded manifold in Euclidean
space, we choose the standard inner product
\begin{equation}\label{canonicmetr}
\langle \bZ_1, \bZ_2\rangle = {\rm tr}\left(\bZ_1^{\top} \bZ_2\right),
\end{equation}
and $\| \bZ\| = \sqrt{\langle \bZ, \bZ\rangle}$. 
An orthogonal projection of a matrix $\bG$ onto the tangent space at a point $\bX$ can be written as
\begin{equation}
\label{GenProjGrad}
\mathcal{P}_{\top}(\bG) =  (\bI - \bX \bX^{\top})\bG +  \dfrac{1}{2}\bX(\bX^{\top} \bG - \bG^{\top}\bX)  = (\bI - \bX \bX^{\top})\bG +  \bX\text{skew}(\bX^{\top} \bG). 
\end{equation}
Therefore, a first-order necessary condition for optimum is
\begin{equation}
\label{FirstProjCond}
(\bI - \bX \bX^{\top})\bG +  \bX\text{skew}(\bX^{\top} \bG)=\bzeros, \ \bC(\bX)=\bzeros. 
\end{equation}
We have the following useful properties
\begin{lemma}
\label{ProjPropLemma}
Let $\bX\in\R ^{n\times p}$ be orthonormal and $\bG\in\R^{n\times p}$ where
\[
(\bI - \bX \bX^\top)\bG + \frac{1}{2}\bX(\bX^\top \bG - \bG^\top \bX) = 0,
\]
then $(\bI - \bX \bX^\top)\bG = 0$ and $\bX^\top \bG - \bG^\top \bX = 0$.
\end{lemma}
\begin{proof}
Let $\bP = \bX \bX^\top$ (projection onto $\mathrm{col}(\bX)$) and $\bP_\perp = \bI - \bX \bX^\top$ (projection onto $\mathrm{col}(\bX)^\perp$) then we have
\[
\bP_\perp \bG + \frac{1}{2}\bX(\bX^\top \bG - \bG^\top \bX) = 0.
\]
Multiply through by $\bP_\perp$ we get
\[
\bP_\perp \bP_\perp \bG + \frac{1}{2}\bP_\perp \bX(\bX^\top \bG - \bG^\top \bX) = 0.
\]
Since $\bP_\perp \bX = 0$ (because $\bX$ is in the column space), this simplifies to
\[
\bP_\perp \bG = 0.
\]
With $\bP_\perp \bG = 0$ we get
\[
\frac{1}{2}\bX(\bX^\top \bG - \bG^\top \bX) = 0
\]
giving
\[
\bX^\top \bX (\bX^\top \bG - \bG^\top \bX) = 0
\]
and since ($\bX^\top \bX = \bI_p$) we get
\[
\bX^\top \bG - \bG^\top \bX = 0.
\]
\end{proof}
 \begin{equation}
\label{eq:constrains}
 \bC=\dfrac{1}{2}\left( \bX^{\top}\bX-\bI \right),
 \end{equation}
we can write the Lagrange function as
$$
\bLagrange ( \bX,\bM) = F(\bX) +  \langle \bC, \bM\rangle,
$$
where $\bM$ is a symmetric matrix with the Lagrange parameters. We obtain
$$
\frac{\partial \bLagrange}{\partial \bX} ( \bX,\bM) = \bG ( \bX)  +  \bX \bM
$$
and the necessary conditions, first-order KKT-conditions, for a minimum are 
\begin{equation}
\label{KKT}
 \bG (\bX)  +  \bX \bM = 0, \ \bC (\bX) = \bzeros.
\end{equation}
A point satisfying these conditions is denoted by $\hat{\bX}, \hat{\bM}$.

We can also attain the condition $(\bI - \bX \bX^{\top})\bG (\bX)   = 0$  considering the  Taylor expansions of $F(\bX)$ and $\bC (\bX)$, i.e.,
$$
F(\bX + \Delta \bX) = F(\bX) + \langle \bG (\bX), \Delta \bX \rangle + \mathcal{O}(\| \Delta \bX \|^2), \
\bC (\bX + \Delta \bX) = \bC (\bX) +  \Delta \bX^{\top}\bX + \bX \Delta \bX^{\top} + \mathcal{O}(\| \Delta \bX \|^2).
$$
The corresponding minimization problem to (\ref{eqas})   is then
\begin{equation*}
\begin{aligned}
     \min_{\Delta \bX} F(\bX) + \langle \bG (\bX), \Delta \bX \rangle + \mathcal{O}(\| \Delta \bX \|^2) \\
    \text{s.t.} \ \Delta \bX^{\top}\bX + \bX \Delta \bX^{\top} + \mathcal{O}(\| \Delta \bX \|^2) = \bzeros.
\end{aligned}
\end{equation*}
With
$$
 \Delta \bX =  (\bI - \bX \bX^{\top}) \Delta \bZ
$$
the constraints are satisfied up to first-order and thus  (\ref{eqas}) is equivalent to
$$
     \min_{\Delta \bZ} F(\bX) + \langle\bG (\bX), (\bI - \bX \bX^{\top})\Delta \bZ \rangle +  \mathcal{O}(\| \Delta \bZ \|^2)
$$
that gives the  first-order necessary optimality condition
\begin{equation}
\label{KKTproj}
(\bI - \bX \bX^{\top})\bG (\bX)   = 0, \ \bC (\bX) = \bzeros.
\end{equation}

As for the second-order optimality conditions we do not consider them in a general form only for our two different examples, the linear eigenvalue problem and the Procrustes problem. 

\section{Damped Dynamical Systems}

\subsection{Lagrange Function Formulation}

Let $\bX = \bX(t)$ and $\bM = \bM (t)$ and consider the damped dynamical system
\begin{equation}
\label{GenDyn}
\ddot{\bX} + \eta \dot{\bX} =    - \bG -   \bX \bM
\end{equation}
where $\eta >0$ is the damping parameter. We note that any stationary point to (\ref{GenDyn}) satisfies the first-order KKT-condition. 
The constraints are satisfied at a stationary point by adding another damped dynamical system in $\bC = \bC(t)$
\begin{equation}
\label{ConDyn}
\ddot{\bC} + \eta {\dot\bC} + \bK \odot \bC = \bzeros,
\end{equation}
where $\odot$ is the Hadamard (elementwise) product and $\bK$ is a matrix with positive elements. It is easy to see that the stationary solution of  (\ref{ConDyn}) is $\bC = \bzeros$.

Using \eqref{eq:constrains}, we obtain
\begin{equation}\label{eq:dC}
\dot{\bC} =\frac{1}{2}\dot{\bX}^{\top}\, \bX + \frac{1}{2}\bX^{\top}\dot{\bX} \mbox{ and }
\ddot{\bC} =\frac{1}{2}\ddot{\bX}^{\top}\, \bX+\dot{\bX}^{\top}\, \dot{\bX}+\frac{1}{2}\bX^{\top}\ddot{\bX}.
\end{equation}
Substituting the equalities above in  \eqref{ConDyn} yields
\begin{equation}
 \bM^{\top} \bX^{\top} \bX+ \bX^{\top} \bX \bM = \bK\odot ( \bX^{\top} \bX- \bI) + 2\dot{\bX}^{\top}\dot{\bX} - \bG^{\top} \bX -\bX^{\top} \bG.
\end{equation}
Next, we use  (\ref{GenDyn}) and get
\begin{equation}
-(\bG + \bX \bM)^{\top}\, \bX +2\dot{\bX}^{\top}\dot{\bX} -  \bX^{\top}(\bG + \bX \bM) +\bK\odot  (\bX^{\top} \bX-\bI)=0
\end{equation}
or
\begin{equation}
\label{eq:aux:1}
 \bM^{\top} \bX^{\top} \bX+ \bX^{\top} \bX \bM = \bK\odot ( \bX^{\top} \bX- \bI) + 2\dot{\bX}^{\top}\dot{\bX} - \bG^{\top} \bX -\bX^{\top} \bG.
\end{equation}
We introduce
\begin{equation}
\label{eq:T}    
 \bT = \bK\odot ( \bX^{\top} \bX- \bI)- \bG^{\top} \bX +2\dot{\bX}^{\top}\dot{\bX}- \bX^{\top} \bG,
\end{equation}
and rewrite \eqref{eq:aux:1} as
\begin{equation*} 
 \bM^{\top} \bX^{\top}\bX+ \bX^{\top} \bX \bM= \bT
\end{equation*}
or since the right-hand side is symmetric 
\begin{equation} 
\label{Syl}
 \bX^{\top} \bX \bM+ \bM \bX^{\top} \bX= \bT
\end{equation}
which is the symmetric Sylvester equation. 
Note that the Lagrange parameters can be found with a Sylvester equation like (\ref{Syl}) for any $F$ since only  $\bT$ will differ (through different $\bG$).
In conclusion, we have the dynamical system
\begin{eqnarray}
\ddot{\bX} + \eta \dot{\bX} = -\bG - \bX \bM,\label{GeneralDamp1}\\
\bX^\top \bX \bM + \bM \bX^\top \bX = \bT, \label{GeneralDamp2}\\
\bT = \mathbf{K} \odot (\bX^\top \bX - \mathbf{I}) - \mathbf{G}^\top \bX + 2 \dot{\bX}^\top \dot{\bX} - \bX^\top \mathbf{G}.\label{GeneralDamp3}
\end{eqnarray}

Let us now assume that the damping in (\ref{GenDyn}) and (\ref{ConDyn}) is not the same, i.e.,
\begin{equation}
\label{GenDyn2}
\ddot{\bX} + \eta_{\bX}\dot{\bX} =    - \bG -   \bX \bM
\end{equation}
and
\begin{equation}
\label{ConDyn2}
\ddot{\bC} + \eta_{\bC} {\dot\bC} + \bK \odot \bC = \bzeros,
\end{equation}
where $\eta_{\bX}>0, \ \eta_{\bC}>0$.

We can use (\ref{ConDyn2}) to find the Lagrange parameters. Consider again
$$
\dot{\bC} = \frac{1}{2} \dot{\bX}^{\top}\, \bX + \frac{1}{2} \bX^{\top}\dot{\bX},\
\ddot{\bC} = \frac{1}{2} \ddot{\bX}^{\top}\, \bX+ \dot{\bX}^{\top}\, \dot{\bX}+ \frac{1}{2} \bX^{\top}\ddot{\bX}
$$
giving
\begin{equation}
\ddot{\bX}^{\top}\, \bX+2\dot{\bX}^{\top}\, \dot{\bX}+\bX^{\top}\ddot{\bX} + \eta_{\bC} (\dot{\bX}^{\top}\, \bX +\bX^{\top}\dot{\bX}) +\bK\odot  (\bX^{\top} \bX-\bI) =0.
\end{equation}
Using  (\ref{GenDyn2}) we get  $\ddot{\bX} = -\eta_{\bX} \dot{\bX}     - \bG -   \bX \bM$ and
\begin{equation}
(-\eta_{\bX} \dot{\bX}     - \bG -   \bX \bM)^{\top}\, \bX +2\dot{\bX}^{\top}\dot{\bX}  + \bX^{\top}(-\eta_{\bX} \dot{\bX}     - \bG -   \bX \bM) + \eta_{\bC} (\dot{\bX}^{\top}\, \bX +\bX^{\top}\dot{\bX}) +\bK\odot  (\bX^{\top} \bX-\bI)=0
\end{equation}
or
\begin{equation}
 \bM^{\top} \bX^{\top} \bX+ \bX^{\top} \bX \bM =  \bT
\end{equation}
where
$$
 \bT = \bK\odot ( \bX^{\top} \bX- \bI)- \bG^{\top} \bX - \bX^{\top} \bG  +2\dot{\bX}^{\top}\dot{\bX} +  (\eta_{\bC}-\eta_{\bX}) (\dot{\bX}^{\top}\, \bX +\bX^{\top}\dot{\bX}).
$$
We see that the only difference compared to the case of equal damping is an additional term in $\bT$. Obviously, following the same idea we could introduce two positive definite damping matrices, say $\bGamma, \bUpsilon$, giving the additional term in $\bT$ on the form 
$(\bGamma-\bUpsilon) (\dot{\bX}^{\top}\, \bX +\bX^{\top}\dot{\bX})$.

\subsection{Projected Gradient Formulation}

Here, we consider (\ref{FirstProjCond}) and define our dynamical system as
\begin{equation}
\label{GenDynProj}
\ddot{\bX} + \eta \dot{\bX} =    -\mathcal{P}_{\top}(\bG), \  \bC (\bX) = \bzeros
\end{equation}
and
\begin{equation}
\label{DynProj}
\ddot{\bX} + \eta \dot{\bX} =    -(\bI - \bX \bX^{\top})\bG, \  \bC (\bX) = \bzeros.
\end{equation}
Furthermore, we note that at a critical point  on the Stiefel manifold where  we have
$$
\bX \bM =  -\dfrac{1}{2}(\bX \bG^{\top} \bX + \bX \bX^{\top} \bG) 
$$
so  $\bX \bM = -\bX \bX^{\top} \bG$ showing that, in this case,  any stationary solution  (\ref{GenDynProj}) or (\ref{DynProj}) is a stationary solution to the Lagrange formulation (\ref{GenDyn}).

\section{The linear eigenvalue problem}
\label{sec:LinEig}

We consider the linear eigenvalue problem of the form $\bA \bX = \bX \bLambda$
where $\bA \in  \mathbb{R}^{n \times n} $ is a positive definite matrix and $\bLambda$ is a diagonal matrix with eigenvalues $\lambda_i, i = 1, \ldots, n$.
A minimization formulation of this for finding the eigenvectors corresponding to the $p$ smallest eigenvalues is 
\begin{equation}\label{mineig}
\min_{\bX\in {\rm St}(n,p)}
F(\bX),\ F(\bX) = \dfrac{1}{2} {\rm tr} (\bX^{\top}\bA \bX).
\end{equation}
We will assume that the eigenvalues are sorted such that
 the $p$ smallest eigenvalues are  $\lambda_j, j=1, \ldots, p$ with corresponding eigenvectors $\mathbf{v}_j, j=1, \ldots, p$. Note that the solution of (\ref{eqas}), say $\hat{\bX}$, has columns that span the eigenspace $\left\{\mathbf{v}_j \right\}_{j=1}^p$ and $F(\hat{\bX}) = \sum_{j=1}^p \lambda_j$.
 
It is easily seen that $\bG(\bX) = \bA \bX$ and
$$
 \bT = \bK\odot ( \bX^{\top} \bX- \bI)- 2\bX^{\top} \bA \bX + 2\dot{\bX}^{\top}\dot{\bX}
$$
giving the dynamical system corresponding to the Lagrange formulation as
\[
\begin{array}{l}
\ddot{\bX} + \eta \dot{\bX} = -\bA \bX  - \bX \bM,\\
\bX^\top \bX \bM + \bM \bX^\top \bX = \bT, \\
\bT = \bK\odot ( \bX^{\top} \bX- \bI)- 2\bX^{\top} \bA \bX +2\dot{\bX}^{\top}\dot{\bX}.
\end{array}
\]
For the projected gradient method we have instead the considerably simpler system
\[
\ddot{\bX} + \eta \dot{\bX} = -(\bI -  \bX\bX^{\top})\bA \bX,\ \bX \in {\rm St}(n,p).
\]
The eigenvalues of $\hat{\bM}$ are 
$0 \geq \mu_1, \geq \ldots \geq \mu_p$ corresponding to the $p$ smallest eigenvalues of $\bA$ with corresponding eigenvectors $\mathbf{u}_i, i=1, \ldots ,p$ (the sign of $\mu_i$ is a consequence of the definition of the Lagrange function). 
As indicated before for this case we have at a stationary solution 
$$
-\bA \hat{\bX}- \hat{\bX} \hat{\bM} = -\bA \hat{\bX} -\hat{\bX}\hat{\bX}^{\top}\bA \hat{\bX} = -(\bI -  \hat{\bX}\hat{\bX}^{\top})\bA \hat{\bX}
$$
which shows  that the two dynamical systems have the same stationary solutions.

\section{The orthogonal Procrustes Problem}
\label{sec:Proc}

The orthogonal Procrustes problem is a matrix approximation problem in linear algebra. In its classical form, one is given two matrices $\bA, \bB$
and wants  to find an orthogonal matrix $\bX$
that minimizes the distance between $\bA \bX$ and $\bB$.
We formulate this in the Frobenius norm, i.e., solving (\ref{eqas}) where
$$
F(\bX) = \dfrac{1}{2}\|\bA\bX - \bB \|_F^2, \bG ( \bX) =   \bA^\top (\bA \bX - \mathbf{B}), 
$$
$\bA\in \R^{m\times n}, \bB\in \R^{m\times p},$ and
the norm $\|\cdot\|_F$ is the Frobenius norm. 

The dynamical system for the Lagrange approach looks like
\[
\begin{array}{l}
\ddot{\bX} + \eta \dot{\bX} = -\bA^\top (\bA \bX - \mathbf{B}) - \bX \bM,\\
\bX^\top \bX \bM + \bM \bX^\top \bX = \bT, \\
\bT = \mathbf{K} \odot (\bX^\top \bX - \mathbf{I}) - \mathbf{G}^\top \bX + 2 \dot{\bX}^\top \dot{\bX} - \bX^\top \mathbf{G},  \\
\mathbf{G} = \bA^\top (\bA \bX - \mathbf{B})
\end{array}
\]

The projected gradient system given by (\ref{GenDynProj}) is
$$
\ddot{\bX} + \eta \dot{\bX} =    -(\bI - \bX \bX^{\top})\bG (\bX) , \  \bC (\bX) = \bzeros,
$$
where
$\bG ( \bX) =   \bA^\top (\bA \bX - \mathbf{B})$.

\section{Asymptotic Stability}
\label{AsStabSec}

In this section we collect the, well known, results for asymptotic stability that we use later on.

\begin{definition}[Asymptotic Stability]
Consider the nonlinear dynamical system 
\[
\dot{\mathbf{z}} = \mathbf{g}(\mathbf{z}),
\]
where \(\mathbf{z} \in \mathbb{R}^q\) and \(\mathbf{g}: \mathbb{R}^q \to \mathbb{R}^q\) is a smooth vector field. An equilibrium point \(\hat{\mathbf{z}}\) (i.e., \(\mathbf{g}(\hat{\mathbf{z}}) = \mathbf{0}\)) is said to be \textit{asymptotically stable} if:
\begin{enumerate}
    \item It is \textit{stable} in the sense of Lyapunov: For every \(\epsilon > 0\), there exists \(\delta > 0\) such that if \(\|\mathbf{z}(0) - \hat{\mathbf{z}}\| < \delta\), then \(\|\mathbf{z}(t) - \hat{\mathbf{z}}\| < \epsilon\) for all \(t \geq 0\).
    \item It is \textit{attractive}: There exists \(\delta' > 0\) such that if \(\|\mathbf{z}(0) - \hat{\mathbf{z}}\| < \delta'\), then \(\lim_{t \to \infty} \mathbf{z}(t) = \hat{\mathbf{z}}\).
\end{enumerate}
\end{definition}

\begin{theorem}[Asymptotic Stability for Nonlinear Systems]
Let \(\hat{\mathbf{z}}\) be an equilibrium point of the system \(\dot{\mathbf{z}} = \mathbf{g}(\mathbf{z})\), where \(\mathbf{g}\) is continuously differentiable in a neighborhood of \(\hat{\mathbf{z}}\). Assume the following:
\begin{enumerate}
    \item The Jacobian matrix \(\mathbf{J}(\hat{\mathbf{z}}) = \left.\frac{\partial \mathbf{g}}{\partial \mathbf{z}}\right|_{\mathbf{z} = \hat{\mathbf{z}}}\) has all eigenvalues with strictly negative real parts.
    \item The higher-order terms of \(\mathbf{g}(\mathbf{z})\) are negligible near \(\hat{\mathbf{z}}\), i.e., \(\mathbf{g}(\mathbf{z}) = \mathbf{J}(\hat{\mathbf{z}})(\mathbf{z} - \hat{\mathbf{z}}) + \mathcal{O}(\|\mathbf{z} - \hat{\mathbf{z}}\|^2)\).
\end{enumerate}
Then, the equilibrium point \(\hat{\mathbf{z}}\) is asymptotically stable.
\end{theorem}

\begin{proof}
For a proof of this theorem, see Theorem 4.7 in the book~\cite{Khalil2002}.
The proof relies on Lyapunov's indirect method and linearization techniques.
\end{proof}

\begin{theorem}[Asymptotic Stability for Constrained Systems]
\label{Theorem:AsymConstr}
Let \(\hat{\mathbf{z}}\) be an equilibrium point of the system \(\dot{\mathbf{z}} = \mathbf{g}(\mathbf{z})\) subject to \(\mathbf{c}(\mathbf{z}) = \mathbf{0}\), where
\begin{itemize}
    \item \(\mathbf{g}: \mathbb{R}^q \to \mathbb{R}^q\) and \(\mathbf{c}: \mathbb{R}^q \to \mathbb{R}^m\) are continuously differentiable in a neighborhood of \(\hat{\mathbf{z}}\).
    \item The constraint Jacobian \(\mathbf{C}(\hat{\mathbf{z}}) = \left.\frac{\partial \mathbf{c}}{\partial \mathbf{z}}\right|_{\mathbf{z} = \hat{\mathbf{z}}}\) has full row rank.
\end{itemize}
Define the reduced Jacobian \(\mathbf{J}_r(\hat{\mathbf{z}})\) as the restriction of \(\mathbf{J}(\hat{\mathbf{z}}) = \left.\frac{\partial \mathbf{g}}{\partial \mathbf{z}}\right|_{\mathbf{z} = \hat{\mathbf{z}}}\) to the tangent space of the constraint manifold at \(\hat{\mathbf{z}}\), i.e., the subspace \(\{\mathbf{v} \in \mathbb{R}^q \mid \mathbf{C}(\hat{\mathbf{z}})\mathbf{v} = \mathbf{0}\}\).  

If all eigenvalues of \(\mathbf{J}_r(\hat{\mathbf{z}})\) have strictly negative real parts, then \(\hat{\mathbf{z}}\) is asymptotically stable under the constraint \(\mathbf{c}(\mathbf{z}) = \mathbf{0}\).
\end{theorem}

\begin{proof}
By the Implicit Function Theorem, the regularity condition ensures that near \(\hat{\mathbf{z}}\), the constraint \(\mathbf{c}(\mathbf{z}) = \mathbf{0}\) defines a smooth \((n-m)\)-dimensional manifold \(\mathcal{M}\). The tangent space at \(\hat{\mathbf{z}}\) is \(\nul \mathbf{C}(\hat{\mathbf{z}})\), with $\nul$ denoting the null space.
Project the system dynamics onto \(\mathcal{M}\) using local coordinates. Let \(\mathbf{z} = \hat{\mathbf{z}} + \mathbf{F}\mathbf{y} + \mathbf{o}(\|\mathbf{y}\|)\), where \(\mathbf{F}\) is a basis for \(\nul \mathbf{C}(\hat{\mathbf{z}})\), and \(\mathbf{y} \in \mathbb{R}^{n-m}\). The reduced dynamics are:
\[
\dot{\mathbf{y}} = \mathbf{F}^\top \mathbf{g}(\hat{\mathbf{z}} + \mathbf{F}\mathbf{y}) + \mathbf{o}(\|\mathbf{y}\|).
\]
Linearizing around \(\mathbf{y} = \mathbf{0}\) gives:
\[
\dot{\mathbf{y}} = \underbrace{\mathbf{F}^\top \mathbf{J}(\hat{\mathbf{z}}) \mathbf{F}}_{\mathbf{J}_r(\hat{\mathbf{z}})} \mathbf{y} + \mathbf{o}(\|\mathbf{y}\|).
\] 
By assumption, \(\mathbf{J}_r(\hat{\mathbf{z}})\) has eigenvalues with negative real parts. Thus, the linearized reduced system is asymptotically stable. 
Apply Lyapunov’s indirect method to the reduced system. Since the higher-order terms \(\mathbf{o}(\|\mathbf{y}\|)\) are negligible near \(\mathbf{y} = \mathbf{0}\), asymptotic stability of the linearized system implies asymptotic stability of the full constrained system.
\end{proof}

\section{First-order Systems}
\label{SysSec}

The aim of this section is to find the Jacobians in Section \ref{AsStabSec} for further asymptotic stability analysis. 
We do so by first presenting a perturbation w.r.t. $\bX$ and $\bV$ around a stationary point $\hat{\bX},\hat{\bY}$. The first-order perturbations are then vectorized in order to find the Jacobian and finally the eigenvalues are derived together with the asymptotic stability results.

\subsection{Eigenvalue Structure for the Lagrange Approach}

\label{SysDamp}

By formulating a first-order system and linearizing together with vectorization we detect the matrix whose eigenvalues are of interest. Note that the analysis below is not general but applicable to our applications.
Consider the system (\ref{GeneralDamp1})--(\ref{GeneralDamp3})
and define the mapping
\begin{equation}
\label{BigSysXV}
\left(
\begin{array}{l}
\dot{\bX}\\
\dot{\bV}
\end{array}
\right) =
\bPhi
\left(
\left(
\begin{array}{l}
\bX \\
\bV
\end{array}
\right) 
\right) =
\left(
\begin{array}{c}
\bV\\
- \eta \bV  -\bG - \bX \bM
\end{array}
\right).
\end{equation}
The first-order term of the right-hand side in (\ref{BigSysXV}) can be written as
$$
\left(
\begin{array}{c}
\Delta \bV\\
- \eta \Delta \bV  - \text{\rm d}\bG (\Delta \bX) - \Delta \bX \bM -\bX \text{\rm d}\bM (\Delta \bX, \Delta \bV)
\end{array}
\right).
$$
We make the following assumption that will be valid for our applications that appear later.
\begin{assumption}
\begin{align}
 \vect (\text{\rm d}\bG (\Delta \bX) + \Delta \bX \bM +\bX  \Delta \bM) = \\
 (\bI \otimes \bJ_{\bX}) \vect (\Delta \bX) + 
  (\bM \otimes \bI) \vect (\Delta \bX)+
   \bJ_{\bM}\vect (\Delta \bX)+
   \bJ_{\bV} \vect (\Delta \bV).
\end{align}
\end{assumption}
We then have 
\begin{align}
\vect \left(
\left(
\begin{array}{c}
\Delta \bV\\
- \eta \Delta \bV  - \text{\rm d}\bG (\Delta \bX) - \Delta \bX \bM -\bX \text{\rm d}\bM (\Delta \bX, \Delta \bV)
\end{array}
\right)  
\right) = \\
\left(
\begin{array}{cc}
\bzeros & \bI \\
- (\bI \otimes \bJ_{\bX})- (\bM \otimes \bI)-  \bJ_{\bM} & -\eta \bI - \bJ_{\bV}
\end{array}
\right) 
   \left(
\begin{array}{l}
\vect(\Delta \bX) \\
\vect(\Delta \bV)
\end{array}
\right).
\end{align}
At a stationary point $\hat{\bV}= \bzeros, \hat{\bX}$ we will in our applications have $\hat{\bJ}_{\bV}=\bzeros$ and we get
\begin{align}
\vect \left(
\left(
\begin{array}{c}
\Delta \bV\\
- \eta \Delta \bV  - \text{\rm d}\bG (\Delta \bX) - \Delta \bX \bM -\bX \text{\rm d}\bM (\Delta \bX, \Delta \bV)
\end{array}
\right)  
\right) = \\
\left(
\begin{array}{cc}
\bzeros & \bI \\
- \hat{\bJ} & -\eta \bI
\end{array}
\right) 
   \left(
\begin{array}{l}
\vect(\Delta \bX) \\
\vect(\Delta \bV)
\end{array}
\right),
\end{align}
where 
$$
\hat{\bJ} =  (\bI \otimes \bJ_{\bX})+ (\bM \otimes \bI)+  \bJ_{\bM}.
$$

Define
$$
\hat{\bJ}_{\text{sys}} =
\left(
\begin{array}{cc}
\bzeros & \bI \\
-\hat{\bJ} & -\eta \bI 
\end{array}
\right) 
$$
and assume that the eigenvalues of $\hat{\bJ}$ are $\alpha_i, i=1, \ldots (np)^2$ then the eigenvalues of the matrix $\hat{\bJ}_{\text{sys}}$ are 
$$
\beta_i = -\dfrac{\eta}{2} \pm \sqrt{\dfrac{\eta^2}{4} - \alpha_i}.
$$
A necessary and sufficient condition for asymptotic convergence is that 
$$
 \Re (\beta_i) \leq 0
$$
or equivalently
$$
 \Re ( \alpha_i ) \geq 0.
$$
The last relation above implies that a sufficient condition for the system to be asymptotically stable is that the eigenvalues of the matrix $\hat{\bJ}$ are non-negative. This will be fully investigated for the applications appearing later.

\subsection{Eigenvalue Structure  when using Projection}

\label{SysProj}

Turning to the projected methods we need to consider that $\bX$ and $\bV$ are independent variables and therefore explicitly state that 
$\bV  \in  T_\bX {\rm St}(n,p)$. We then get the
DAE system
\begin{equation}
\label{bigsys1}
\begin{array}{l}
\dot{\bX} = \bV\\
\dot{\bV}=  - \eta \bV    -(\bI - \bX \bX^{\top})\bG\\
\bX  \in {\rm St}(n,p), \ \bV \in  T_\bX {\rm St}(n,p).
\end{array}
\end{equation}
In order to analyse the asymptotic stability we need to find an expression for 
 $\bF$ in Section \ref{AsStabSec}. We know that the tangent space of the constraint manifold w.r.t. $\bX$ is  $T_\bX {\rm St}(n,p)$. Since the tangent space is a canonical isomorphism the  tangent space of the constraint manifold $T_\bX {\rm St}(n,p)$ is itself. 
 
Consider the mapping
\begin{equation}
\label{BigMat}
\left(
\begin{array}{l}
\dot{\bX}\\
\dot{\bV}
\end{array}
\right) =
\bPhi
\left(
\left(
\begin{array}{l}
\bX \\
\bV
\end{array}
\right) 
\right) =
\left(
\begin{array}{c}
\bV\\
- \eta \bV -(\bI - \bX \bX^{\top})\bG
\end{array}
\right) 
\end{equation}
where we note that $(\bI - \bX \bX^{\top})\bG$ does not depend on $\bV$.
Following similar notation as in the previous section we define
\begin{equation}
\label{Jsys}
\hat{\bJ}_{\text{sys}} =
\left(
\begin{array}{cc}
\bzeros & \bI \\
-\hat{\bJ} & -\eta \bI 
\end{array}
\right) 
\end{equation}
as the Jacobian of the vectorized system at a stationary point $\hat{\bV}= \bzeros, \hat{\bX}$, where $\hat{\bJ}$  is the Jacobian of $(\bI - \bX \bX^{\top})\bG$ w.r.t. 
$$
    \left(
\begin{array}{l}
\vect(\bX) \\
\vect(\bV)
\end{array}
\right).
$$
Define 
$$
\bP = \bX \bX^{\top}
$$
and
$$
\bP_\perp = \bI -\bX \bX^{\top},
$$
then if $\bX  \in {\rm St}(n,p)$, $\bP$ is an orthogonal projection onto $\text{range} (\bX)$ and $\bP_\perp$ is an orthogonal projection onto $T_\bX {\rm St}(n,p)$. The projection onto $T_\bX {\rm St}(n,p)$ in the vectorized space is then
$$
\bI \otimes \bP_\perp
$$
giving
\begin{align}
\bF = 
\left(
\begin{array}{cc}
\bI \otimes \bP_\perp & \bzeros \\
\bzeros &\bI \otimes \bP_\perp 
\end{array}
\right) 
\end{align}
and
\begin{align}
\bF^{\top}
\hat{\bJ}_{\text{sys}}
\bF=
\left(
\begin{array}{cc}
\bI \otimes \bP_\perp & \bzeros \\
\bzeros &\bI \otimes \bP_\perp 
\end{array}
\right) 
\left(
\begin{array}{cc}
\bzeros & \bI \\
-\hat{\bJ} & -\eta \bI 
\end{array}
\right) 
\left(
\begin{array}{cc}
\bI \otimes \bP_\perp & \bzeros \\
\bzeros &\bI \otimes \bP_\perp 
\end{array}
\right) =\\
\left(
\begin{array}{cc}
\bzeros & (\bI \otimes \bP_\perp) \\
-(\bI \otimes \bP_\perp)\hat{\bJ}(\bI \otimes \bP_\perp) & -\eta (\bI \otimes \bP_\perp) 
\end{array}
\right). 
\end{align}
For later use we note that if  $\bX  \in {\rm St}(n,p)$
$$
\bP_\perp \bX  = (\bI -\bX \bX^{\top})\bX = \bzeros,
\bX^{\top}\bP_\perp = \bX^{\top}(\bI -\bX \bX^{\top}) = \bzeros
$$
or in vectorized space
$$
(\bI \otimes \bP_\perp) \vect(\bX)  =  \bzeros,
\vect(\bX)^{\top}(\bI \otimes \bP_\perp)= \bzeros
$$
and that the eigenvalue problem 
$$
\left(
\begin{array}{cc}
\bzeros & (\bI \otimes \bP_\perp) \\
-(\bI \otimes \bP_\perp)\hat{\bJ}(\bI \otimes \bP_\perp) & -\eta (\bI \otimes \bP_\perp) 
\end{array}
\right) 
   \left(
\begin{array}{l}
\vect(\bX) \\
\vect(\bV)
\end{array}
\right) = 
\kappa 
   \left(
\begin{array}{l}
\vect(\bX) \\
\vect(\bV)
\end{array}
\right)
$$
can be rewritten as
$$
\left(
\begin{array}{cc}
\bzeros &  \bI\\
-(\bI \otimes \bP_\perp)\hat{\bJ}(\bI \otimes \bP_\perp) & -\eta \bI
\end{array}
\right) 
   \left(
\begin{array}{l}
\vect(\bX) \\
(\bI \otimes \bP_\perp)\vect(\bV)
\end{array}
\right) = 
\kappa 
   \left(
\begin{array}{l}
\vect(\bX) \\
(\bI \otimes \bP_\perp)\vect(\bV)
\end{array}
\right)
$$
which has a matrix of the same structure as in (\ref{Jsys}).

\section{The Eigenvalue Problem for the Lagrange formulation}

Consider the eigenvalue problem with
$$
F(\bX) = \dfrac{1}{2} {\rm tr} (\bX^{\top}\bA \bX), \ \bG(\bX) = \bA \bX
$$
and
$$
 \bT = \bK\odot ( \bX^{\top} \bX- \bI)- 2\bX^{\top} \bA \bX + 2\dot{\bX}^{\top}\dot{\bX}
$$
giving
\begin{eqnarray}
\ddot{\bX} + \eta \dot{\bX} = -\bA \bX  - \bX \bM,\label{EigDampSys1}\\
\bX^\top \bX \bM + \bM \bX^\top \bX = \mathbf{T}, \label{EigDampSys2}\\
\bT = \bK\odot ( \bX^{\top} \bX- \bI)- 2\bX^{\top} \bA \bX +2\dot{\bX}^{\top}\dot{\bX}.\label{EigDampSys3}
\end{eqnarray}
At the stationary solution we have
$$
 \hat{\bM} = \dfrac{1}{2}\hat{\bT} =    -\hat{\bX}^{\top} \bA \hat{\bX}.
$$

\subsection{Eigenvalues of the Jacobian and Asymptotic Stability}

To simplify the analysis we first formulate the system (\ref{EigDampSys1})-(\ref{EigDampSys3})
using a similarity transformation of $\hat{\bM}$ as
$$
   \hat{\bM} = \bQ \hat{\bOmega} \bQ^{\top}, \, \hat{\bOmega} = \text{diag}(\mu_1, \ldots , \mu_p), \,  \bQ^{\top}  \bQ = \bI, \mu_1 \geq  \cdots  \geq\mu_p.
$$
Note that $-\mu_i \geq 0$ are the $p$ smallest eigenvalues of $\bA$. Define
$$
    \bY =\bX  \bQ, \, \bM_{\bQ} = \bQ^{\top} \bM  \bQ, \, \bT_{\bQ} = \bQ^{\top}\bT\bQ
$$
giving
$$
\hat{\bM}_{\bQ} = \hat{\bOmega} = \dfrac{1}{2}\hat{\bT}_{\bQ} = -\hat{\bY}^{\top} \bA \hat{\bY}, \, \bA \hat{\bY} =  -\hat{\bY} \hat{\bOmega}.
$$
In other words, $\hat{\bM}_{\bQ}$ will have eigenvalues $\mu_i$ and eigenvectors as unit vectors $\be_i$.

We now transform the system (\ref{EigDampSys1})-(\ref{EigDampSys3}) (change of basis for the column space of $\bX$) into the equivalent system
\begin{eqnarray}
\ddot{\bY} + \eta \dot{\bY} = -\bA \bY  - \bY \bM_{\bQ},\label{EigDampSysQ1}\\
\bY^\top \bY \bM_{\bQ} + \bM_{\bQ} \bY^\top \bY = \bT_{\bQ},\label{EigDampSysQ2} \\
\bT_{\bQ} = \bQ^{\top}\left( \bK\odot \left( \bQ( \bY^{\top} \bY- \bI)\right) \bQ^{\top} \right)\bQ- 2\bY^{\top} \bA \bY +2\dot{\bY}^{\top}\dot{\bY}.\label{EigDampSysQ3}
\end{eqnarray}

If we define $\bS(\bY) = \bA \bY  + \bY \bM_{\bQ}(\bY) $ we want to show that the eigenvalues of  the matrix $\bJ$ in
$$
\vect( \bS(\hat{\bY} + \Delta \bY) - \bS(\hat{\bY}) ) = \bJ(\hat{\bY})\vect(\Delta \bY ) + \mathcal{O} (\| \Delta \bY \|_F^2)
$$
are all non-negative, see Section \ref{SysSec}.
We begin by stating a lemma.
\begin{lemma}
\label{DeltaSLemma}
$$
\bS(\hat{\bY} + \Delta \bY) - \bS(\hat{\bY}) = \bA\, \Delta \bY  + \Delta \bY\, \hat{\bOmega}+ \hat{\bY}\, \Delta \bM_{\bQ}
$$

$$
2\,  \hat{\bY}\Delta  \bM_{\bQ}  =  \hat{\bY}\, \Delta \bT_{\bQ} -   \hat{\bY}\hat{\bOmega} \, \Delta \bY^{\top} \hat{\bY} -  \hat{\bY}\hat{\bOmega} \hat{\bY}^{\top}   \, \Delta \bY  -   \hat{\bY}\, \Delta \bY^{\top}\hat{\bY} \hat{\bOmega} -  \hat{\bY}\hat{\bY}^{\top}\, \Delta \bY \hat{\bOmega}
$$
\begin{eqnarray}
 \hat{\bY}\, \Delta \bT_{\bQ} =  \hat{\bY}\bQ^{\top}\left(\bK\odot \bQ( \, \Delta\bY^{\top}\,  \hat{\bY} +  \hat{\bY}^{\top}\,   \, \Delta\bY) \bQ^{\top} \right)\bQ- 2 \hat{\bY}\, \Delta\bY^{\top}  \hat{\bY}\hat{\bOmega} -  2\hat{\bY}\hat{\bY}^{\top} \bA  \, \Delta\bY
\end{eqnarray}
and
$$
\begin{array}{l}
\bS(\hat{\bY} + \Delta \bY) - \bS(\hat{\bY}) = 
\bA\, \Delta \bY  + \Delta \bY\, \hat{\bOmega}+ \\
\dfrac{1}{2}
\left( 
\hat{\bY}\, \Delta \bT_{\bQ} -   \hat{\bY}\hat{\bOmega} \, \Delta \bY^{\top} \hat{\bY} -  \hat{\bY}\hat{\bOmega} \hat{\bY}^{\top}   \, \Delta \bY  -   \hat{\bY}\, \Delta \bY^{\top}\hat{\bY} \hat{\bOmega} -  \hat{\bY}\hat{\bY}^{\top}\, \Delta \bY \hat{\bOmega} 
\right) =
\end{array}
$$
$$
\begin{array}{l}
\bA\, \Delta \bY  + \Delta \bY\, \hat{\bOmega}+ \\
\dfrac{1}{2}
\left( 
\hat{\bY}\bQ^{\top}\left(\bK\odot \bQ( \, \Delta\bY^{\top}\,  \hat{\bY} +  \hat{\bY}^{\top}\,   \, \Delta\bY) \bQ^{\top} \right)\bQ
-2 \hat{\bY}\, \Delta\bY^{\top}  \hat{\bY}\hat{\bOmega} 
-  2\hat{\bY}\hat{\bY}^{\top} \bA  \, \Delta\bY
\right.\\
\left.
-   \hat{\bY}\hat{\bOmega} \, \Delta \bY^{\top} \hat{\bY} -  \hat{\bY}\hat{\bOmega} \hat{\bY}^{\top}   \, \Delta \bY  
-   \hat{\bY}\, \Delta \bY^{\top}\hat{\bY} \hat{\bOmega} 
-  \hat{\bY}\hat{\bY}^{\top}\, \Delta \bY \hat{\bOmega} 
\right).
\end{array}
$$

\end{lemma}

\begin{proof}
$$
\bS(\hat{\bY} + \Delta \bY) - \bS(\hat{\bY}) = \bA\, \Delta \bY  + \Delta \bY\, \hat{\bOmega}+ \bY\, \Delta \bM_{\bQ}
$$

$$
2\,  \hat{\bY}\Delta  \bM_{\bQ}  =  \hat{\bY}\, \Delta \bT_{\bQ} -   \hat{\bY}\hat{\bM}_{\bQ} \, \Delta \bY^{\top} \hat{\bY} -  \hat{\bY}\hat{\bM}_{\bQ} \hat{\bY}^{\top}   \, \Delta \bY  -  \,  \hat{\bY}\Delta \bY^{\top}\hat{\bY} \hat{\bM}_{\bQ} -  \hat{\bY}\hat{\bY}^{\top}\, \Delta \bY \hat{\bM}_{\bQ}
$$
\begin{eqnarray}
\, \Delta \bT_{\bQ} = \bQ^{\top}\left(\bK\odot \bQ( \, \Delta\bY^{\top}\,  \hat{\bY} + \hat{\bY}^{\top}\,   \, \Delta\bY) \bQ^{\top} \right)\bQ- 2(\, \Delta\bY^{\top} \bA \hat{\bY} +  \hat{\bY}^{\top} \bA  \, \Delta\bY)
\end{eqnarray}
\end{proof}
The following lemma will help when vectorizing the result above.
\begin{lemma}
\label{KronLemma0}
Assume that the matrices $\bE, \bF, \bH$ have sizes such that the product $\bE \bF \bH$ is well defined then
$$
\begin{array}{l}
 \vect(\bE \bF \bH) = (\bH^{\top} \otimes \bE)\vect(\bF), \
  \vect(\bE \bF^{\top} \bH) = (\bH^{\top} \otimes \bE)\bN \vect(\bF), \
 (\bE \otimes \bF)\bN = \bN(\bF \otimes \bE),
 \end{array}
$$
where $\bN$ is the commutation matrix, see \cite{wikicomm}.
\end{lemma}
We now apply Lemma \ref{KronLemma0} to the products in Lemma \ref{DeltaSLemma}.
\begin{lemma}
\label{KronLemma}
The terms in Lemma \ref{DeltaSLemma} satisfies the following.
$$
\begin{array}{l}
 \vect(\bA\, \Delta \bY) = (\bI \otimes \bA) \vect( \Delta \bY)\\
 \vect(\Delta \bY\, \bOmega) = (\bOmega \otimes \bI)  \vect( \Delta \bY)\\
\vect(\hat{\bY}\bOmega \, \Delta \bY^{\top} \hat{\bY}) = (\bY^{\top} \otimes  \hat{\bY}\bOmega)\bN \vect( \Delta \bY)\\
\vect(\hat{\bY}\bOmega \hat{\bY}^{\top}   \, \Delta \bY) =(\bI \otimes  \hat{\bY}\bOmega \hat{\bY}^{\top} ) \vect( \Delta \bY)\\
\vect(\hat{\bY}\, \Delta \bY^{\top}\hat{\bY} \bOmega) = 
(\Omega \hat{\bY}^{\top}  \otimes \hat{\bY})\bN \vect( \Delta \bY)
\\
\vect(\hat{\bY}\hat{\bY}^{\top}\, \Delta \bY \bOmega) = (\Omega \otimes \hat{\bY}\hat{\bY}^{\top})\vect( \Delta \bY)\\

\vect(\hat{\bY}^{\top}\bA \, \Delta \bY) = (\bI \otimes \hat{\bY}\hat{\bY}^{\top}\bA)\vect( \Delta \bY)\\
\vect(\hat{\bY} \, \Delta \bY^{\top} \hat{\bY}) = (\bY^{\top} \otimes  \hat{\bY})\bN \vect( \Delta \bY)\\
\vect(\hat{\bY}\bY^{\top} \, \Delta  \hat{\bY}) = ( \bI \otimes  \hat{\bY}\bY^{\top}) \vect( \Delta \bY),
\end{array}
$$
and acts on the eigenvector $\mathbf{u}_i \otimes \mathbf{v}_j$ as
$$
\begin{array}{l}
 (\bI \otimes \bA)(\mathbf{u}_i \otimes \mathbf{v}_j) = \lambda_j(\mathbf{u}_i \otimes \mathbf{v}_j),i=1,\ldots, p,j=1,\ldots, n\\
 (\bOmega \otimes \bI)  (\mathbf{u}_i \otimes \mathbf{v}_j) =\mu_i(\mathbf{u}_i \otimes \mathbf{v}_j),i=1,\ldots, p,j=1,\ldots, n\\
 (\bY^{\top} \otimes  \hat{\bY}\bOmega)\bN(\mathbf{u}_i \otimes \mathbf{v}_j) =\mu_i (\mathbf{u}_j \otimes \mathbf{v}_i),i=1,\ldots, p,j=1,\ldots, p\\
(\bI \otimes  \hat{\bY}\bOmega \hat{\bY}^{\top} ) (\mathbf{u}_i \otimes \mathbf{v}_j) =\mu_j (\mathbf{u}_i \otimes \mathbf{v}_j),i=1,\ldots, p,j=1,\ldots, p\\
(\Omega \hat{\bY}^{\top}  \otimes \hat{\bY})\bN(\mathbf{u}_i \otimes \mathbf{v}_j) =\mu_j (\mathbf{u}_j \otimes \mathbf{v}_i),i=1,\ldots, p,j=1,\ldots, p
\\
 (\Omega \otimes \hat{\bY}\hat{\bY}^{\top})(\mathbf{u}_i \otimes \mathbf{v}_j) =\mu_i (\mathbf{u}_i \otimes \mathbf{v}_j),i=1,\ldots, p,j=1,\ldots, p\\
(\bI \otimes \hat{\bY}\hat{\bY}^{\top}\bA)(\mathbf{u}_i \otimes \mathbf{v}_j) =\lambda_j (\mathbf{u}_i \otimes \mathbf{v}_j),i=1,\ldots, p,j=1,\ldots, p\\
 (\bY^{\top} \otimes  \hat{\bY})\bN (\mathbf{u}_i \otimes \mathbf{v}_j) =(\mathbf{u}_j \otimes \mathbf{v}_i),i=1,\ldots, p,j=1,\ldots, p\\
 ( \bI \otimes  \hat{\bY}\bY^{\top})(\mathbf{u}_i \otimes \mathbf{v}_j) =(\mathbf{u}_i \otimes \mathbf{v}_j),i=1,\ldots, p,j=1,\ldots, p,
\end{array}
$$
and are zero otherwise,
where
$$
(\mathbf{u}_j \otimes \mathbf{v}_i)^{\top}(\mathbf{u}_i \otimes \mathbf{v}_j) = \delta_{ij},i=1,\ldots, p,j=1,\ldots, n
$$
i.e., $(\mathbf{u}_j \otimes \mathbf{v}_i)$ and $(\mathbf{u}_i \otimes \mathbf{v}_j)$ are orthonormal.
\end{lemma}
For a general $\bK$ we were not able to get further in the analysis so we make the  simplification  $\bK_{ij}=\nu$ where $\nu$ is a positive constant.
\begin{lemma}
\label{SeparateJhat}
Assume that $\bK_{ij}=\nu$ then
$$
\begin{array}{l}
\bJ(\hat{\bY}) =  (\bI \otimes \bA) + (\bOmega \otimes \bI)  
+ \nu\dfrac{1}{2}\left(  \hat{\bY}(\hat{\bY}\otimes \bI) \bN +   \hat{\bY}( \bI \otimes \hat{\bY})  \right)  - \\
(\Omega \hat{\bY}^{\top}  \otimes \hat{\bY})\bN -
(\bI \otimes \hat{\bY}\hat{\bY}^{\top}\bA) -
\dfrac{1}{2}
\left(
 (\bY^{\top} \otimes  \hat{\bY}\bOmega)\bN +
 (\bI \otimes  \hat{\bY}\bOmega \hat{\bY}^{\top} )+
 (\Omega \hat{\bY}^{\top}  \otimes \hat{\bY})\bN+
 (\Omega \otimes \hat{\bY}\hat{\bY}^{\top})
\right)
   \end{array}
$$
and
$$
\begin{array}{l}
\bJ(\hat{\bY})(\mathbf{u}_i \otimes \mathbf{v}_j) =  
\lambda_j (\bI \otimes \bA)(\mathbf{u}_i \otimes \mathbf{v}_j) + 
\mu_i  (\mathbf{u}_i \otimes \mathbf{v}_j)
+ 
\nu\dfrac{1}{2}\left(  
(\mathbf{u}_j \otimes \mathbf{v}_i) + (\mathbf{u}_i \otimes \mathbf{v}_j)  \right)  - \\
\mu_j (\mathbf{u}_j \otimes \mathbf{v}_i)  -
\lambda_j (\mathbf{u}_i \otimes \mathbf{v}_j) -
\dfrac{1}{2}
\left(
\mu_i (\mathbf{u}_j \otimes \mathbf{v}_i)+
\mu_j (\mathbf{u}_i \otimes \mathbf{v}_j)+
 \mu_j (\mathbf{u}_j \otimes \mathbf{v}_i)+
\mu_i  (\mathbf{u}_i \otimes \mathbf{v}_j)
\right).
   \end{array}
$$
\end{lemma}
\begin{theorem}
\label{AsymThmEig}
The dynamical system (\ref{EigDampSys1})-(\ref{EigDampSys3}) with $\bK_{ij} = \nu$ is asymptotically
stable.
\end{theorem}
\begin{proof}
The smallest eigenvalue of  $\bJ(\hat{\bY})$ is given by
$$
\min_\bw \dfrac{\bw^{\top}\bJ(\hat{\bY})\bw}{\bw^{\top} \bw},
$$
where
$$
\bw = \sum_{ij}\alpha_{ij}(\mathbf{u}_i \otimes \mathbf{v}_j) 
$$
and $\|\bw  \| =1$ if $\sum_{ij} \alpha_{ij}^2=1$.
From Lemma \ref{SeparateJhat} we get that
$$
\bw^{\top}\bJ(\hat{\bY})\bw = 
\sum_{i=1}^p \sum_{j=1}^n \alpha_{ij}^2 (\lambda_j + \mu_i) -
\sum_{i=1}^p \sum_{j=1}^p \alpha_{ij}^2 (\lambda_j + \mu_i)+
\sum_{i=1}^p \sum_{j=1}^p  \alpha_{ij}^2\nu -
 \dfrac{1}{2}
 \sum_{i=1}^p \sum_{j=1}^p 
 \alpha_{ij}^2
\left(
 \mu_i + \mu_j + \mu_j + \mu_i
\right)
$$
or simplifying
$$
\bw^{\top}\bJ(\hat{\bY})\bw = 
\sum_{i=1}^p \sum_{j=p+1}^n \alpha_{ij}^2 (\lambda_j + \mu_i) +
 \sum_{i=1}^p \sum_{j=1}^p 
 \alpha_{ij}^2
\left(\nu -\mu_i -\mu_j 
\right).
$$

To minimize $\bw^{\top}\bJ(\hat{\bY})\bw$ under the constraint $\sum_{i=1}^p \sum_{j=1}^n \alpha_{ij}^2 = 1$, we allocate all weight to the term with the smallest coefficient. Define $\lambda^* = \min_{p+1\leq j \leq n}  \lambda_j \geq -\mu_p$ and note that $\min \mu_j = \mu_p<0$, $\max \mu_j = \mu_1<0$ we get the minimum as
$$
\min \left\{\nu - 2\mu_1, \lambda^* + \mu_p \right\} \geq 0.
$$
\end{proof}

\section{The Projected Gradient System for the Eigenvalue Problem}

We start by looking at
\begin{equation}
\label{ProjGradSys}
\begin{array}{l}
\dot{\bX} = \bV\\
\dot{\bV}=  - \eta \bV  -(\bI - \bX \bX^{\top})\bA \bX \\
\bC(\bX)=\bzeros.
\end{array}
\end{equation}
Define
$$
\bS_{\bG} = (\bI - \bX \bX^{\top})\bA \bX 
$$
and consider
$$
\Delta \bS_{\bG} = \bA \, \Delta \bX  -  \Delta\bX \, \hat{\bX}^{\top}\bA \hat{\bX}  - \hat{\bX} \, \Delta\bX^{\top} \, \bA \hat{\bX}  -
\hat{\bX}\hat{\bX}^{\top} \bA \, \Delta \bX =
 \bA \, \Delta \bX  +  \Delta\bX \, \hat{\bM} - \hat{\bX} \, \Delta\bX^{\top} \, \bA \hat{\bX} -
\hat{\bX}\hat{\bX}^{\top} \bA \, \Delta \bX
$$
giving the matrix of interest
$$
  (\bI \otimes \bA) + (\hat{\bM} \otimes \bI) - \left( (\bA \hat{\bX})^{\top} \otimes  \hat{\bX}\right) \bN - 
   (\bI \otimes \hat{\bX}\hat{\bX}^{\top} \bA). 
$$
However, if we assume that we are close to the constraints we should use the projected perturbation for the analysis, i.e., 
$$
\Delta \bX  = (\bI -  \hat{\bX}\hat{\bX}^{\top})\Delta \bZ
$$
that gives
$$
\Delta\bX^{\top} \, \bA \hat{\bX} = \bzeros, \ \hat{\bX}^{\top} \bA \, \Delta \bX = \bzeros
$$
which simplifies the Jacobian as
$$
  (\bI \otimes (\bI -  \hat{\bX}\hat{\bX}^{\top})\bA)  + (\hat{\bM} \otimes  (\bI -  \hat{\bX}\hat{\bX}^{\top})).
$$
\begin{theorem}
  The dynamical system (\ref{ProjGradSys}) is asymptotically stable.
\end{theorem}
\begin{proof}
   We have that
   $$
   \left( (\bI \otimes (\bI -  \hat{\bX}\hat{\bX}^{\top})\bA)  + (\hat{\bM} \otimes  (\bI -  \hat{\bX}\hat{\bX}^{\top})) \right) (\mathbf{u}_i \otimes \mathbf{v}_j) = 
   (\lambda_j + \mu_i) (\mathbf{u}_i \otimes \mathbf{v}_j), j=1,\ldots, n-p, i=1,\ldots, p,
   $$
   i.e., the eigenvalues are $\lambda_j + \mu_i$. Since $-\mu_i$ are the smallest eigenvalues of $\bA$ we conclude that $\lambda_j + \mu_i \geq 0$.
\end{proof}

Alternatively we refer to Theorem \ref{Theorem:AsymConstr} above where we note that all functions are twice continuously differentiable,  the projection is $\bI -  \hat{\bX}\hat{\bX}^{\top}$, and the Jacobian of the constraints has full rank as seen in the Lemma below. 

\begin{lemma}[Rank of the Jacobian of Orthonormality Constraints]
Let $\mathbf{X} \in \mathbb{R}^{n \times p}$ with $p < n$, and define the constraint function $\mathbf{C}(\mathbf{X}) = \mathbf{X}^\top\mathbf{X} - \mathbf{I}_p$. 
Let $\mathbf{g}(\mathbf{z}) = \operatorname{vec}(\mathbf{C}(\mathbf{X}))$ where $\mathbf{z} = \operatorname{vec}(\mathbf{X})$. 
At any point where $\mathbf{C}(\mathbf{X}) = \mathbf{0}$ (i.e., $\mathbf{X}^\top\mathbf{X} = \mathbf{I}_p$), the Jacobian matrix $\mathbf{J}_{\mathbf{g}}(\mathbf{z}) = \frac{\partial \mathbf{g}}{\partial \mathbf{z}}$ has rank:

\[
\operatorname{rank}(\mathbf{J}_{\mathbf{g}}(\mathbf{z})) = \frac{p(p + 1)}{2}.
\]

Moreover, the Jacobian has size $\frac{p(p + 1)}{2} \times np$ and thus it has full row rank.
\end{lemma}

\begin{proof}

Differentiating the constraints gives
\[
d\mathbf{C} = d(\mathbf{X}^\top\mathbf{X}) = (d\mathbf{X})^\top\mathbf{X} + \mathbf{X}^\top d\mathbf{X}.
\]
Vectorizing both sides and using the property $\operatorname{vec}(A^\top) =  \bN_{m,n}\operatorname{vec}(A)$ for $A \in \mathbb{R}^{m \times n}$ (where $ \bN_{m,n}$ is the commutation matrix), we obtain:
\[
\operatorname{vec}(d\mathbf{C}) = \underbrace{\left( (\mathbf{X}^\top \otimes \mathbf{I}_p)  \bN_{n,p} + (\mathbf{I}_p \otimes \mathbf{X}^\top) \right)}_{\mathbf{J}_{\mathbf{g}}} \operatorname{vec}(d\mathbf{X}).
\]
The null space of the Jacobian $\mathbf{J}_{\mathbf{g}}$ consists of tangent vectors to the Stiefel manifold at $\mathbf{X}$. From the differentiating we see that these are matrices $\mathbf{V} \in \mathbb{R}^{n \times p}$ satisfying
\[
\mathbf{X}^\top\mathbf{V} + \mathbf{V}^\top\mathbf{X} = \mathbf{0}
\]
which implies that $\mathbf{X}^\top\mathbf{V}$ is skew-symmetric. The dimension of skew symmetric matrices  is $\frac{p(p - 1)}{2}$ and thus, by the rank-nullity theorem,
\[
\operatorname{rank}(\mathbf{J}_{\mathbf{g}}) = np - \dim(\nul(\mathbf{J}_{\mathbf{g}})) = np - \frac{p(p - 1)}{2} = \frac{p(p + 1)}{2}.
\]
Finally, since $\mathbf{C}(\mathbf{X})$ is symmetric, there are $\frac{p(p + 1)}{2}$ independent constraints. Thus $\mathbf{J}_{\mathbf{g}}$ has $\frac{p(p + 1)}{2}$ rows and $np$ columns.
\end{proof}

\section{The Procrustes Problem using the Lagrange Formulation}

Consider the unbalanced Procrustes problem, see \cite{EldenPark1999}, with
$$
F(\bX) = \dfrac{1}{2} \| \bA \bX - \bB \|_F^2, \ \bG(\bX) = \bA^{\top}(\bA  \bX - \bB)
$$
and
\begin{eqnarray}
\ddot{\bX} + \eta \dot{\bX} = -\bG - \bX \bM,\label{ProcDampSys1}\\
\bG = \bA^{\top}(\bA  \bX - \bB)\label{ProcDampSys2}\\
\bX^\top \bX \bM + \bM \bX^\top \bX = \mathbf{T}, \label{ProcDampSys3}\\
\bT = \bK\odot ( \bX^{\top} \bX- \bI)- \bG^{\top} \bX +2\dot{\bX}^{\top}\dot{\bX}- \bX^{\top} \bG.\label{ProcDampSys4}
\end{eqnarray}
At the stationary solution we have
$$
 \hat{\bM} = \dfrac{1}{2}\hat{\bT} =    -\dfrac{1}{2}(\hat{\bG}^{\top} \hat{\bX} + \hat{\bX}^{\top}\hat{\bG}),
$$
or since
$\hat{\bG}^{\top} \hat{\bX} = \hat{\bX}^{\top}\hat{\bG}$ , see \cite{BIRTEA2020102868},  
$$
\hat{\bM} =-\hat{\bG}^{\top} \hat{\bX} = -\hat{\bX}^{\top}\hat{\bG}.
$$

\subsection{Eigenvalues of the Jacobian and Asymptotic Stability}

We can do the same stability analysis here since $\bG$ is a linear function of $\bX$. The convergence will depend on the eigenvalues of $\bA^{\top} \bA$ and $\hat{\bM}$ and from \cite{EldenPark1999} we have that a second-order condition for a minimum is
\begin{equation}
\label{EldenParkProc}
\lambda_{j}(\bA^{\top} \bA) + \mu_i \geq 0, i=1, \ldots, p, j = p+1, \ldots, n.
\end{equation}
\begin{lemma}
The non-zero eigenvalues of $-\hat{\bX}\hat{\bG}^{\top}$ are the same as for $\hat{\bM}$.
\end{lemma}
\begin{proof}
$$
 \hat{\bX}\hat{\bG}^{\top} \bw = \gamma \bw \Longleftrightarrow \hat{\bG}\hat{\bX}^{\top} \bw = \gamma \bw \Longleftrightarrow 
 \hat{\bX}^{\top}\hat{\bG}\hat{\bX}^{\top} \bw = \gamma \hat{\bX}^{\top}\bw  \Longleftrightarrow 
 \hat{\bX}^{\top}\hat{\bG}\bz = \gamma \bz  
$$
and we now have to show that $ \hat{\bX}\hat{\bG}^{\top}$ is symmetric.

Since \(\bG = \bX \bX^T \bG\) and  \(\bX \bX^T\) is a projection onto the column space of \(\bX\), there exists a $\bQ \in \mathbb{R}^{p\times p}$ such that
   \[
   \bG = \bX \bQ.
   \]
   In \(\bG^T \bX = \bX^T \bG\), insert \(\bG = \bX \bQ\) giving
   \[
   \bG^T \bX = \bQ^T \bX^T \bX = \bQ^T,
   \bX^T \bG = \bX^T \bX \bQ = \bQ
   \]
   showing that \(\bQ^T = \bQ\). Now compute \(\bX \bG^T\) and \(\bG \bX^T\) using \(\bG = \bX \bQ\) and \(\bQ^T = \bQ\), i.e., 
   \[
   \bX \bG^T = \bX \bQ^T \bX^T = \bX \bQ \bX^T,
   \bG \bX^T = \bX \bQ \bX^T
   \]
   which proves that \(\bX \bG^T = \bG \bX^T\).
\end{proof}
Following the analysis for the eigenvalue problem we define $\bS(\bY) = \bA \bY  + \bY \bM_{\bQ}(\bY) $
giving
\begin{align}
\bS(\hat{\bY} + \Delta \bY) - \bS(\hat{\bY}) = 
\bA^{\top}\bA\, \Delta \bY  + \Delta \bY\, \bOmega+ \\
\dfrac{1}{2}
(
\hat{\bY}\bQ^{\top}\left(\bK\odot \bQ( \, \Delta\bY^{\top}\,  \hat{\bY} +  \hat{\bY}^{\top}\,   \, \Delta\bY) \bQ^{\top} \right)\bQ- \\
 \hat{\bY}  \, \Delta \bY^{\top}\bA^{\top}\bA \hat{\bY} -
   \hat{\bY} (\hat{\bY} \bOmega)^{\top}  \, \Delta\bY  -
  \hat{\bY}\, \, \Delta \bY ^{\top} \hat{\bY} \bOmega   -  
 \hat{\bY}\hat{\bY}^{\top} \bA^{\top}\bA\, \Delta \bY\\
-   \hat{\bY}\bOmega \, \Delta \bY^{\top} \hat{\bY} -  \hat{\bY}\bOmega \hat{\bY}^{\top}   \, \Delta \bY  -   \hat{\bY}\, \Delta \bY^{\top}\hat{\bY} \bOmega -  \hat{\bY}\hat{\bY}^{\top}\, \Delta \bY \bOmega 
) .
\end{align}

\begin{lemma}
\label{KronLemma2}
Kronecker rules applied
$$
\begin{array}{l}
 \vect(\bA^{\top}\bA\, \Delta \bY) = (\bI \otimes \bA^{\top}\bA) \vect( \Delta \bY)\\
 \vect(\Delta \bY\, \bOmega) = (\bOmega \otimes \bI)  \vect( \Delta \bY)\\
 \vect(\hat{\bY}\, \Delta \bY^{\top}\bA^{\top}\bA \hat{\bY} ) = 
( \hat{\bY}^{\top}\bA^{\top}\bA   \otimes \hat{\bY})\bN \vect( \Delta \bY)\\
\vect(\hat{\bY}\bOmega \, \Delta \bY^{\top} \hat{\bY}) = (\bY^{\top} \otimes  \hat{\bY}\bOmega)\bN \vect( \Delta \bY)\\
\vect(\hat{\bY}\bOmega \hat{\bY}^{\top}   \, \Delta \bY) =(\bI \otimes  \hat{\bY}\bOmega \hat{\bY}^{\top} ) \vect( \Delta \bY)\\
\vect(\hat{\bY}\, \Delta \bY^{\top}\hat{\bY} \bOmega) = 
(\Omega \hat{\bY}^{\top}  \otimes \hat{\bY})\bN \vect( \Delta \bY)
\\
\vect(\hat{\bY}\hat{\bY}^{\top}\, \Delta \bY \bOmega) = (\Omega \otimes \hat{\bY}\hat{\bY}^{\top})\vect( \Delta \bY)\\
\vect(\hat{\bY}^{\top}\bA^{\top}\bA \, \Delta \bY) = (\bI \otimes \hat{\bY}\hat{\bY}^{\top}\bA)\vect( \Delta \bY)\\
\vect(\hat{\bY} \, \Delta \bY^{\top} \hat{\bY}) = (\bY^{\top} \otimes  \hat{\bY})\bN \vect( \Delta \bY)\\
\vect(\hat{\bY}\bY^{\top} \, \Delta  \hat{\bY}) = ( \bI \otimes  \hat{\bY}\bY^{\top}) \vect( \Delta \bY)
\end{array}
$$
and with the eigenvector $\mathbf{u}_i \otimes \mathbf{v}_j$
$$
\begin{array}{l}
 (\bI \otimes \bA^{\top}\bA)(\mathbf{u}_i \otimes \mathbf{v}_j) = \lambda_j(\mathbf{u}_i \otimes \mathbf{v}_j),\ i=1,\ldots, p,j=1,\ldots, n\\
 (\bOmega \otimes \bI)  (\mathbf{u}_i \otimes \mathbf{v}_j) =\mu_i(\mathbf{u}_i \otimes \mathbf{v}_j),\ ,i=1,\ldots, p,j=1,\ldots, n\\
 ( \hat{\bY}^{\top}\bA^{\top}\bA   \otimes \hat{\bY})\bN (\mathbf{u}_i \otimes \mathbf{v}_j) =
 \lambda_j (\mathbf{u}_j \otimes \mathbf{v}_i),\ i=1,\ldots, p,j=1,\ldots, p\\
 (\bY^{\top} \otimes  \hat{\bY}\bOmega)\bN(\mathbf{u}_i \otimes \mathbf{v}_j) =\mu_i (\mathbf{u}_j \otimes \mathbf{v}_i),\ i=1,\ldots, p,j=1,\ldots, p\\
(\bI \otimes  \hat{\bY}\bOmega \hat{\bY}^{\top} ) (\mathbf{u}_i \otimes \mathbf{v}_j) =\mu_j (\mathbf{u}_i \otimes \mathbf{v}_j),\ i=1,\ldots, p,j=1,\ldots, p\\
(\Omega \hat{\bY}^{\top}  \otimes \hat{\bY})\bN(\mathbf{u}_i \otimes \mathbf{v}_j) =\mu_j (\mathbf{u}_j \otimes \mathbf{v}_i),\ i=1,\ldots, p,j=1,\ldots, p
\\
 (\Omega \otimes \hat{\bY}\hat{\bY}^{\top})(\mathbf{u}_i \otimes \mathbf{v}_j) =\mu_i (\mathbf{u}_i \otimes \mathbf{v}_j),\ i=1,\ldots, p,j=1,\ldots, p\\
(\bI \otimes \hat{\bY}\hat{\bY}^{\top}\bA^{\top}\bA)(\mathbf{u}_i \otimes \mathbf{v}_j) =\lambda_j (\mathbf{u}_i \otimes \mathbf{v}_j),\ i=1,\ldots, p,j=1,\ldots, p\\
 (\bY^{\top} \otimes  \hat{\bY})\bN (\mathbf{u}_i \otimes \mathbf{v}_j) =(\mathbf{u}_j \otimes \mathbf{v}_i),\ i=1,\ldots, p,j=1,\ldots, p\\
 ( \bI \otimes  \hat{\bY}\bY^{\top})(\mathbf{u}_i \otimes \mathbf{v}_j) =(\mathbf{u}_i \otimes \mathbf{v}_j),\ i=1,\ldots, p,j=1,\ldots, p
\end{array}
$$
and zero otherwise where
$$
(\mathbf{u}_j \otimes \mathbf{v}_i)^{\top}(\mathbf{u}_i \otimes \mathbf{v}_j) = \delta_{ij},
$$
i.e., $(\mathbf{u}_j \otimes \mathbf{v}_i)$ and $(\mathbf{u}_i \otimes \mathbf{v}_j)$ are orthonormal.
\end{lemma}
By using (\ref{EldenParkProc})
we get the following theorem which has a proof almost identical to the one for Theorem \ref{AsymThmEig} and is therefore omitted.
\begin{theorem}
The dynamical system (\ref{ProcDampSys1})-(\ref{ProcDampSys4}) with $\bK_{ij} = \nu$ is asymptotically
stable.
\end{theorem}

\section{The Procrustes Problem using the Projection Formulation}

Define
$$
\bS_{\bG} = (\bI - \bX \bX^{\top})\bG
$$
and make the perturbation giving 
$$
\Delta \bS_{\bG} =\Delta \bG  -  \Delta\bX \, \hat{\bX}^{\top}\hat{\bG}  - \hat{\bX} \, \Delta\bX^{\top}\hat{\bG}    -
\hat{\bX}\hat{\bX}^{\top} \, \Delta \bG,\ \Delta \bG = \bA^{\top} \bA\,  \Delta\bX. 
$$
Assume that 
$$
\Delta \bX  = (\bI -  \hat{\bX}\hat{\bX}^{\top})\Delta \bZ
$$
we get
$$
\Delta \bS_{\bG} =\Delta \bG  -  (\bI -  \hat{\bX}\hat{\bX}^{\top})\Delta \bZ \, \hat{\bX}^{\top}\hat{\bG}  - \hat{\bX} \, \Delta \bZ^{\top}  (\bI -  \hat{\bX}\hat{\bX}^{\top})\hat{\bG}  -
\hat{\bX}\hat{\bX}^{\top} \, \Delta \bG 
$$
or since $ (\bI -  \hat{\bX}\hat{\bX}^{\top})\hat{\bG} = \bzeros$
$$
\Delta \bS_{\bG} =\Delta \bG  -  (\bI -  \hat{\bX}\hat{\bX}^{\top})\Delta \bZ \, \hat{\bX}^{\top}\hat{\bG}    -
\hat{\bX}\hat{\bX}^{\top} \, \Delta \bG.
$$
Inserting $\bA^{\top} \bA\,  \Delta\bX$ we have 
$$
\Delta \bS_{\bG}  =\bA^{\top} \bA\,  \Delta\bX  -  (\bI -  \hat{\bX}\hat{\bX}^{\top})\Delta \bZ \, \hat{\bX}^{\top}\hat{\bG}    -
\hat{\bX}\hat{\bX}^{\top} \, \bA^{\top} \bA\,  \Delta\bX
$$
and with $\Delta \bX  = (\bI -  \hat{\bX}\hat{\bX}^{\top})\Delta \bZ$ 
$$
\Delta \bS_{\bG} =\bA^{\top} \bA\,  (\bI -  \hat{\bX}\hat{\bX}^{\top})\Delta \bZ  -  (\bI -  \hat{\bX}\hat{\bX}^{\top})\Delta \bZ \, \hat{\bX}^{\top}\hat{\bG}    -
\hat{\bX}\hat{\bX}^{\top} \, \bA^{\top} \bA\,  (\bI -  \hat{\bX}\hat{\bX}^{\top})\Delta \bZ
$$
or collecting terms
$$
\Delta \bS_{\bG} =(\bI -  \hat{\bX}\hat{\bX}^{\top})\bA^{\top} \bA\,  (\bI -  \hat{\bX}\hat{\bX}^{\top})\Delta \bZ  -  (\bI -  \hat{\bX}\hat{\bX}^{\top})\Delta \bZ \, \hat{\bX}^{\top}\hat{\bG}    
$$
with the corresponding Jacobian
$$
\bJ(\hat{\bX}) =  (\bI \otimes (\bI -  \hat{\bX}\hat{\bX}^{\top})\bA^{\top} \bA\,  (\bI -  \hat{\bX}\hat{\bX}^{\top}))   - 
(\hat{\bX}^{\top}\hat{\bG} \otimes  (\bI -  \hat{\bX}\hat{\bX}^{\top})).
$$
Note that if $\hat{\bX}^{\top}\hat{\bG}=\hat{\bG}^{\top}\hat{\bX}$ we get $\hat{\bX}^{\top}\hat{\bG} = -\hat{\bM}$ and
$$
\bJ(\hat{\bX}) =  (\bI \otimes (\bI -  \hat{\bX}\hat{\bX}^{\top})\bA^{\top} \bA\,  (\bI -  \hat{\bX}\hat{\bX}^{\top}))   +
(\hat{\bM} \otimes  (\bI -  \hat{\bX}\hat{\bX}^{\top})).
$$

\begin{lemma}
Let $\hat{\mathbf{X}}$ be a real $n \times p$ matrix ($n > p$) with orthonormal columns, $\mathbf{A}$ a real $m \times n$ matrix, and $\hat{\mathbf{M}}$ a symmetric real $p \times p$ matrix. Define the matrix $\mathbf{J}(\hat{\mathbf{X}})$ as:
\[
\mathbf{J}(\hat{\mathbf{X}}) = \left(\mathbf{I}_p \otimes (\mathbf{I}_n - \hat{\mathbf{X}}\hat{\mathbf{X}}^\top) \mathbf{A}^\top \mathbf{A} (\mathbf{I}_n - \hat{\mathbf{X}}\hat{\mathbf{X}}^\top)\right) + \left(\hat{\mathbf{M}} \otimes (\mathbf{I}_n - \hat{\mathbf{X}}\hat{\mathbf{X}}^\top)\right),
\]
where $\otimes$ denotes the Kronecker product. Then, the eigenvalues of $\mathbf{J}(\hat{\mathbf{X}})$ are:

\begin{enumerate}
    \item $\lambda_i + \mu_j \geq 0$ for all $i = 1, \dots, p$ and $j = 1, \dots, r$, where:
    \begin{itemize}
        \item $\lambda_i$ are the eigenvalues of $\hat{\mathbf{M}}$,
        \item $\mu_j$ are the non-zero eigenvalues of $(\mathbf{I}_n - \hat{\mathbf{X}}\hat{\mathbf{X}}^\top) \mathbf{A}^\top \mathbf{A} (\mathbf{I}_n - \hat{\mathbf{X}}\hat{\mathbf{X}}^\top)$,
        \item $r = \operatorname{rank}\left((\mathbf{I}_n - \hat{\mathbf{X}}\hat{\mathbf{X}}^\top) \mathbf{A}^\top \mathbf{A} (\mathbf{I}_n - \hat{\mathbf{X}}\hat{\mathbf{X}}^\top)\right)$.
    \end{itemize}
    \item Zero with multiplicity $p \times (n - r)$.
\end{enumerate}
\end{lemma}

\begin{proof}
Let $\mathbf{P}_\perp = \mathbf{I}_n - \hat{\mathbf{X}}\hat{\mathbf{X}}^\top$, which is the orthogonal projection onto the complement of the column space of $\hat{\mathbf{X}}$. Since $\hat{\mathbf{X}}$ has orthonormal columns, $\mathbf{P}_\perp$ is symmetric and idempotent ($\mathbf{P}_\perp^2 = \mathbf{P}_\perp$), with eigenvalues 1 (multiplicity $n - p$) and 0 (multiplicity $p$). 
   Since $\hat{\mathbf{M}}$ is symmetric, it has an orthonormal eigenbasis $\{\mathbf{v}_i\}_{i=1}^p$ with eigenvalues $\{\lambda_i\}_{i=1}^p$:
   \[
   \hat{\mathbf{M}} = \sum_{i=1}^p \lambda_i \mathbf{v}_i \mathbf{v}_i^\top.
   \]

   The matrix $\mathbf{P}_\perp \mathbf{A}^\top \mathbf{A} \mathbf{P}_\perp$ is symmetric and positive semidefinite. Let its non-zero eigenvalues be $\{\mu_j\}_{j=1}^r$ with corresponding orthonormal eigenvectors $\{\mathbf{w}_j\}_{j=1}^r$ in the range of $\mathbf{P}_\perp$. The remaining $n - r$ eigenvectors $\{\mathbf{w}_j\}_{j=r+1}^n$ span the null space (including the null space of $\mathbf{P}_\perp$).

We now analyze the action of $\mathbf{J}(\hat{\mathbf{X}})$ on vectors of the form $\mathbf{v}_i \otimes \mathbf{w}_j$.

First consider the case that $\mathbf{w}_j$ is in the null space of $\mathbf{P}_\perp \mathbf{A}^\top \mathbf{A} \mathbf{P}_\perp$ ($j > r$) then
 $\mathbf{P}_\perp \mathbf{w}_j = \mathbf{0}$ and thus
\[
\mathbf{J}(\hat{\mathbf{X}}) (\mathbf{v}_i \otimes \mathbf{w}_j) = (\mathbf{I}_p \otimes \mathbf{P}_\perp \mathbf{A}^\top \mathbf{A} \mathbf{P}_\perp)(\mathbf{v}_i \otimes \mathbf{w}_j) + (\hat{\mathbf{M}} \otimes \mathbf{P}_\perp)(\mathbf{v}_i \otimes \mathbf{w}_j) = \mathbf{0} + \mathbf{0} = \mathbf{0}.
\]
This contributes $p \times (n - r)$ zero eigenvalues (since $j = r+1, \dots, n$).

Now consider  $\mathbf{w}_j$ in the range of $\mathbf{P}_\perp \mathbf{A}^\top \mathbf{A} \mathbf{P}_\perp$ ($j \leq r$).
Since $\mathbf{w}_j$ is an eigenvector of $\mathbf{P}_\perp \mathbf{A}^\top \mathbf{A} \mathbf{P}_\perp$ with eigenvalue $\mu_j$ and $\mathbf{P}_\perp \mathbf{w}_j = \mathbf{w}_j$ (because $\mathbf{w}_j$ is in the range of $\mathbf{P}_\perp$) we have
\[
\mathbf{J}(\hat{\mathbf{X}}) (\mathbf{v}_i \otimes \mathbf{w}_j) = (\mathbf{I}_p \mathbf{v}_i) \otimes (\mathbf{P}_\perp \mathbf{A}^\top \mathbf{A} \mathbf{P}_\perp \mathbf{w}_j) + (\hat{\mathbf{M}} \mathbf{v}_i) \otimes (\mathbf{P}_\perp \mathbf{w}_j) = \mathbf{v}_i \otimes (\mu_j \mathbf{w}_j) + (\lambda_i \mathbf{v}_i) \otimes \mathbf{w}_j = (\lambda_i + \mu_j)(\mathbf{v}_i \otimes \mathbf{w}_j).
\]
This gives eigenvalues $\lambda_i + \mu_j$ for $i = 1, \dots, p$ and $j = 1, \dots, r$.

The null space of $\mathbf{P}_\perp \mathbf{A}^\top \mathbf{A} \mathbf{P}_\perp$ has dimension $n - r$. For each $\mathbf{v}_i$, there are $n - r$ independent $\mathbf{w}_j$ in this null space, leading to $p \times (n - r)$ zero eigenvalues. 

Finally, we have from the second-order necessary conditions for a local minima, see \cite{BIRTEA2020102868} that 
$$
\langle \bA^{\top} \bA \mathbf{P}_\perp \, \Delta \bZ , \mathbf{P}_\perp \, \Delta \bZ \rangle - 
 \langle \mathbf{P}_\perp \, \Delta \bZ \hat{\bX}^{\top}\hat{\bG},  \mathbf{P}_\perp \, \Delta \bZ\rangle \geq \bzeros.
$$
\end{proof}

\section{The symplectic Euler}

The simplest, and most often adequate, method for solving the dynamical systems we have introduced is the symplectic Euler method, see, e.g., \cite{symp}. We will use this method both for the projected gradient and when using a Lagrangian formulation. 
The purpose of our numerical experiments is not to find the most efficient implementation but to show some interesting aspects of our approach.
The derivation of optimal parameters is left for further research. 

As an introduction to our use of symplectic Euler consider 
\begin{equation}
\ddot{\bX} + \eta \dot{\bX} =    - \bPsi
\end{equation}
where $\bPsi$ is different depending on the actual application. First, the system is formulated as a first-order system
\begin{equation}
\begin{array}{l}
\dot{\bX} = \bV\\
\dot{\bV} = - \eta \bV - \bPsi
\end{array}
\end{equation}
and discretized in time using a time step $h$ as
\begin{equation}
\label{SymplEulerPsi}
\begin{array}{l}
\bV_{k+1} = (1-\eta h)\bV_k - h \bPsi (\bX_k,\bV_k)\\
\bX_{k+1} = \bX_k + h \bV_{k+1}.
\end{array}
\end{equation}

\subsection{Damped Constraints}

Looking at~(\ref{GeneralDamp1})--(\ref{GeneralDamp3}) and applying the symplectic Euler
we get the general discrete system
\begin{equation}
\begin{array}{l}
\bT_k = \bK\odot ( \bX_k ^{\top} \bX_k - \bI)- \bG_k ^{\top} \bX_k  +2\bV_k^{\top}\bV_k- \bX_k^{\top} \bG_k\\
\text{Solve } \bX_k^{\top} \bX_k \bM + \bM \bX_k^{\top}\bX_k =  \bT_k \text{ for }  \bM_{k}\\
\bV_{k+1} = \bV_k - h \left(\bG_k + \bX_k \bM_k+ \eta \bV_k \right), \\
\bX_{k+1} = \bX_k + h \bV_{k+1}.
\end{array}
\end{equation}
In particular, for the  linear eigenvalue problem we have
\begin{equation}
\begin{array}{l}
\bG_k = \bA  \bX_k \\
\bT_k = \bK\odot ( \bX_k ^{\top} \bX_k - \bI)- 2\bX_k^{\top} \bA \bX_k  +2\bV_k^{\top}\bV_k\\
\text{Solve } \bX_k^{\top} \bX_k \bM + \bM \bX_k^{\top}\bX_k =  \bT_k \text{ for }  \bM_{k}\\
\bV_{k+1} = \bV_k - h \left(\bA \bX_k + \bX_k \bM_k + \eta \bV_k \right), \\
\bX_{k+1} = \bX_k + h \bV_{k+1},
\end{array}
\end{equation}
and for the Procrustes problem 
\begin{equation}
\begin{array}{l}
\bG_k = \bA^{\top}(\bA\bX_k-\bB)\\
\bT_k = \bK\odot ( \bX_k ^{\top} \bX_k - \bI)- \bG_k ^{\top} \bX_k  +2\bV_k^{\top}\bV_k- \bX_k^{\top} \bG_k\\
\text{Solve } \bX_k^{\top} \bX_k \bM + \bM \bX_k^{\top}\bX_k =  \bT_k \text{ for }  \bM_{k}\\
\bV_{k+1} = \bV_k - h \left(\bG_k + \bX_k \bM_k + \eta \bV_k \right), \\
\bX_{k+1} = \bX_k + h \bV_{k+1}.
\end{array}
\end{equation}

\subsection{Projected Gradient}

Consider (\ref{DynProj}) used for solving the eigenvalue problem.
Since we need to be on the manifold we find, in every iteration, the orthonormal matrix closest to $\bX_k$ in Frobenius norm we solve 
$$
\argmin_{\bX\in {\rm St}(n,p)} \| \bX - \bX_{k} \|_F.
$$
 The solution to this problem is, see \cite{GolubVanLoan},
$$
  \bX = \bU_{\bX} \bV_{\bX}^{\top}
$$
where $\bU_{\bX}, \bV_{\bX}$ are the orthonormal matrices in the thin SVD of $\bX_k$. 
Furthermore,  for the eigenvalue problem we can use Lemma
\ref{ProjPropLemma} to simplify the projection of the gradient
giving the discrete system
\begin{equation}
\begin{array}{l}
\bX_{k} = \argmin_{\bX\in {\rm St}(n,p)} \| \bX - \bX_{k} \|_F\\\
\bG_k = \bA  \bX_k\\
\bV_{k+1} = \bV_k - h \left( ( \bI - \bX_k \bX_k ^{\top} ) \bG_k + \eta \bV_k \right), \\
\bX_{k+1} = \bX_k + h \bV_{k+1}.
\end{array}
\end{equation}
Note that $\bV_k$ will stay in the tangent space if we use the zero matrix as starting value.

\begin{figure}[h]
\begin{center}
\hspace{-1cm}
\includegraphics[height=10cm]{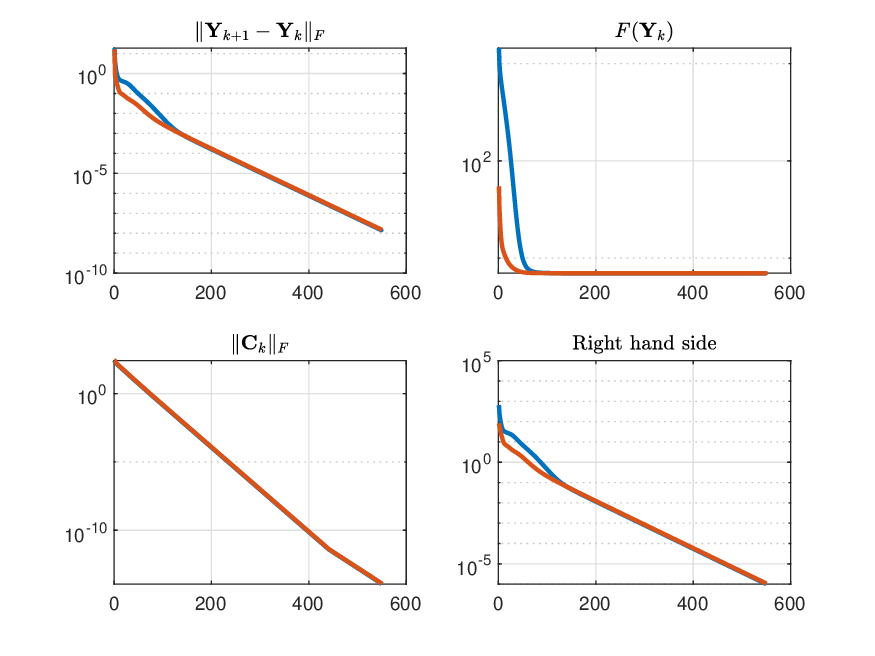}
\end{center}
\caption{Eigenvalue problem. Convergence of $\bY_k$ (see main text) vs iteration steps $k$.
Blue curves are for the Lagrange approach and red for the projection. 'Right-hand-side' is referring to the norm of the right-hand-side in the second-order damped system, respectively. The problem size was $n=100, p=10$ and the initial $\bY_{k=0}$ was chosen as a random matrix with elements uniformly distributed in $\left[-100,100 \right]$.}
\label{fig:EFeig}
\end{figure}

\begin{figure}[ht]
\begin{center}
\hspace{-1cm}
\includegraphics[height=10cm]{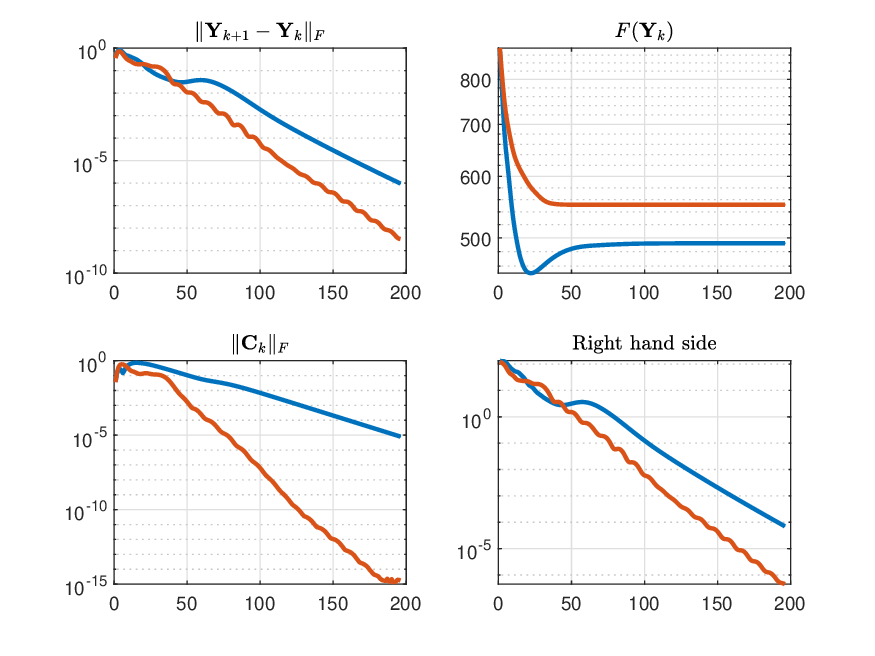}
\end{center}
\caption{Procrustes problem. Convergence of $\bY_k$ (see main text) vs iteration steps $k$. 
The parameters and initial conditions are the same as in Figure~\ref{fig:EFeig}.
Note in the upper-right subfigure that the Lagrange approach finds a lower local minimum.}
\label{fig:EFProc}
\end{figure}

Now consider the Procrustes problem. Here, the projection is given by (\ref{GenProjGrad}) and the discrete system is
\begin{equation}
\begin{array}{l}
\bX_{k} = \argmin_{\bX\in {\rm St}(n,p)} \| \bX - \bX_{k} \|_F\\
\bG_k = \bA^{\top}(\bA \bX_k - \bB)\\
\PROJ (\bG_k) = ( \bI - \bX_k \bX_k ^{\top} )\bG_k + \dfrac{1}{2}\bX_k (\bX_k ^{\top}\bG_k - \bG_k ^{\top}\bX_k)\\
\bV_{k+1} = \bV_k - h \left( \PROJ (\bG_k) + \eta \bV_k \right), \\
\bX_{k+1} = \bX_k + h \bV_{k+1}. \\
\end{array}
\end{equation}

\begin{figure}[ht]
\begin{center}
\hspace{-1cm}
\includegraphics[height=6cm]{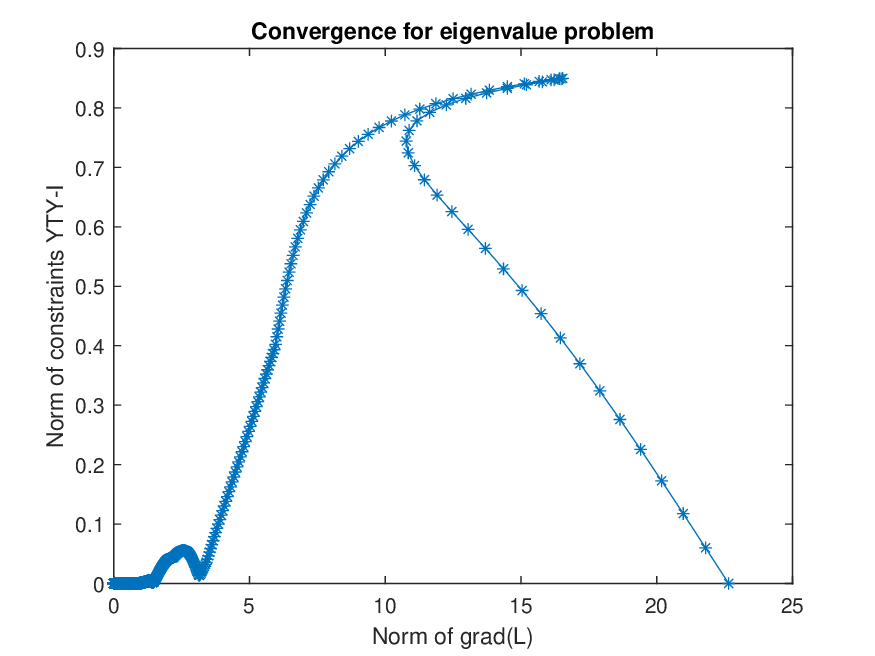}
\end{center}
\caption{Eigenvalue problem. Convergence of $\| \bC \|_F$ vs $\| \bA \bX + \bX \bM \|_F$.}
\label{fig:EigFC1}
\end{figure}

\begin{figure}[ht]
\begin{center}
\hspace{-1cm}
\includegraphics[height=6cm]{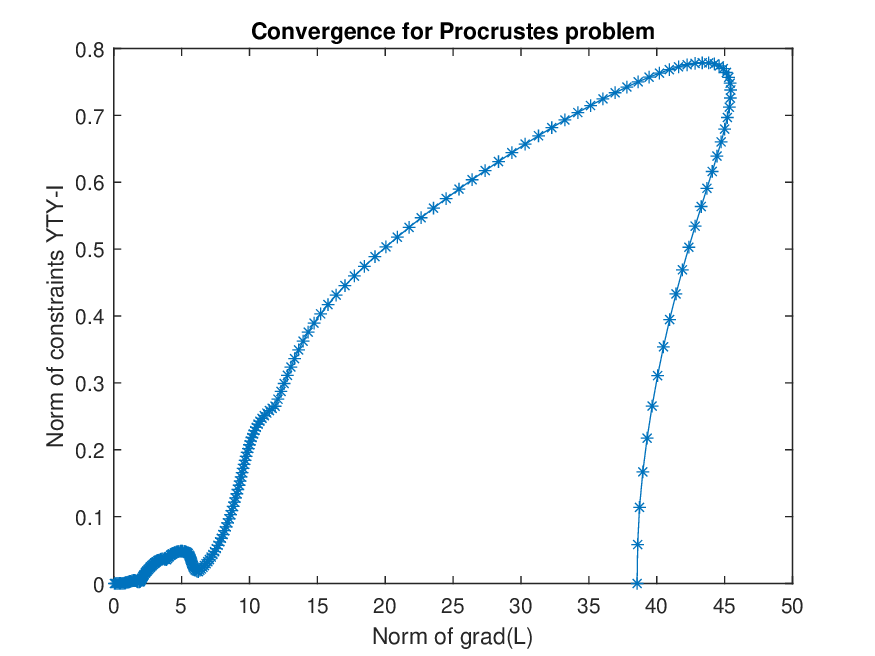}
\end{center}
\caption{Procrustes problem. Convergence of $\| \bC \|_F$ vs $\| \bA \bX + \bX \bM \|_F$.}
\label{fig:EigFC2}
\end{figure}

\section{Numerical examples}

We have divided the numerical examples in three different parts depending on what we want to illustrate. 

Firstly, we show the actual convergence behaviour of  the two approaches with regard to the most relevant parts in the dynamical system, i.e., the solution, the objective function, the constraints, and the gradient of the Lagrange function or the projected gradient.

Secondly, for the Lagrange approach we plot the norm of the constraints and the norm of the Lagrange function in order to illustrate how they depend on each other. 

Finally, we look at a  two-dimensional eigenvalue problem to show how the solution vectors converge to the optimal eigenvector  depending on the eigenvalues of the matrix. Here we focus on the parameter $\nu$ in the damped dynamical system for the constraints.

\subsection{Convergence}

In Figures \ref{fig:EFeig} and \ref{fig:EFProc} we show the convergence of the two approaches, see captions. The step size $h$, damping $\eta$, and the parameter $\nu$ have been chosen in order to get convergence, but not necessarily optimal convergence. From these numerical experiments it is not possible to conclude what approach converges faster, and the conjecture is that it is problem dependent. An interesting feature for the Procrustes problem is that the Lagrange approach seems to almost always find a better (i.e., lower $F$) minimum, but with slower convergence.

\subsection{The two dynamical systems}

Using the Lagrange approach for the two problems, we show how the norm of the constraints and the norm of the projected Lagrangian converge, see Figures \ref{fig:EigFC1} and \ref{fig:EigFC2}. We can notice that there is no global monotonic behaviour in either of these norms especially in the beginning of the iterates. However, close to the solution the norms both decrease monotonically.

\subsection{The linear two-dimensional Eigenvalue Problem}

Let $n=2,p=1$ and assume that $\bA = \text{diag}(\lambda,1), \lambda<1$ giving $\hat{\bX}= \be_1, \hat{\bM}=-\lambda = \mu_1$. Since $p=1$ the matrix $\bK$ will be scalar $\bK=\nu$ and $\bX$ a vector in the plane. Below we show the convergence of $\bX(t_k),k=1, \ldots$ for different $\lambda, \nu$, and $\bX_0$. We have chosen $\bX_0 =  0.8[1+ \cos (\theta),\ \sin(\theta)]^{\top}, \theta = 0, \pi/4, \pi/2, \ldots, 7\pi/4$ and $\bV_0$ has random elements uniformly distributed in $[-0.5,0.5]^2$. 
The time step and damping are  $h=0.1$ and $\eta = 5$, respectively, in all these plots.

Notice that for the projected approach we stay on the manifold (green curves). We include the corresponding points in the plots only as reference with the below comments only for the Lagrange approach.

Furthermore, for this simple problem it's easy to see that the solution of the constraint equations will be
\[
c(t) = e^{-\frac{\eta}{2} t} \left[ \alpha  \cos(\omega t) + \beta \sin(\omega t) \right],
\omega = \frac{\sqrt{4\nu - \eta^2}}{2}
\]
for constants $\alpha, \beta$ depending on the initial values but there is no closed solution for $\bX(t)$. 

Consider Figure~\ref{fig:EigFig1} with $\nu =10$ so that the constraint equation is overdamped ($4\nu - \eta^2>0$).
It is clearly seen that convergence towards the manifold is fast and independent of the size of the smallest eigenvalue.

\begin{figure}[ht]
\begin{center}
\hspace{-1cm}
\includegraphics[height=6cm]{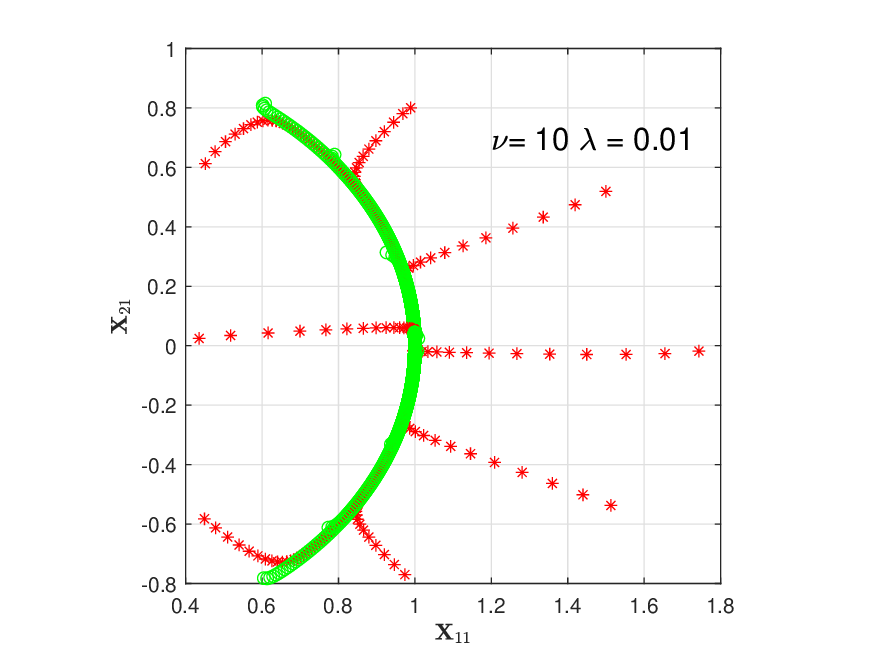}
\includegraphics[height=6cm]{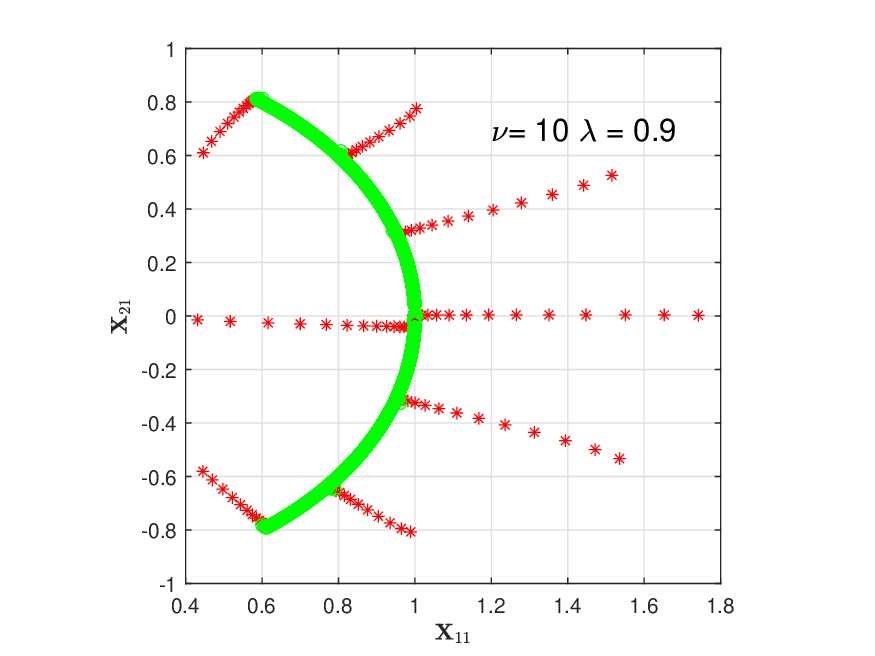}
\end{center}
\caption{A large $\nu$ gives fast convergence towards the manifold.}
\label{fig:EigFig1}
\end{figure}

With a small $\nu=0.1$ and eigenvalue $\lambda = 0.01$ as in Figure \ref{fig:EigFig3} we observe faster convergence to the second component in the corresponding eigenvector. 
This is an underdamped case of the constraint equation.

\begin{figure}[ht]
\begin{center}
\hspace{-1cm}
\includegraphics[height=6cm]{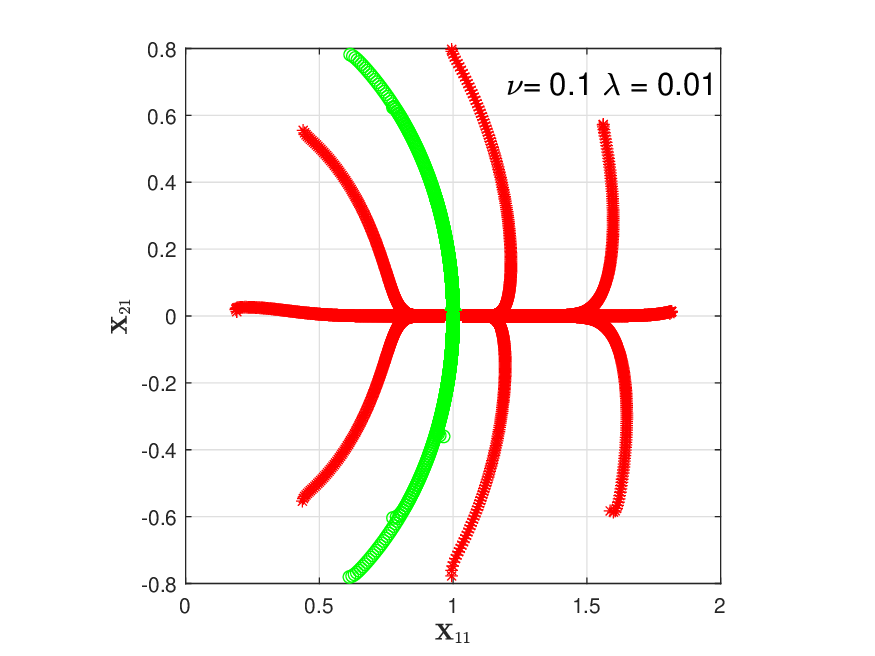}
\end{center}
\caption{A small $\nu$ and a small $\lambda$ give convergence towards $\bX_{21}=0$.}
\label{fig:EigFig3}
\end{figure}

Finally, for a small $\nu$ and a large $\lambda$, again underdamped, the trajectories do not seem to converge via any subspace. This is an interesting case since one possible advantage of the Lagrange approach is to be able to have trajectories taking 'shortcuts' outside the manifold given by the constraints.

\begin{figure}[ht]
\begin{center}
\hspace{-1cm}
\includegraphics[height=6cm]{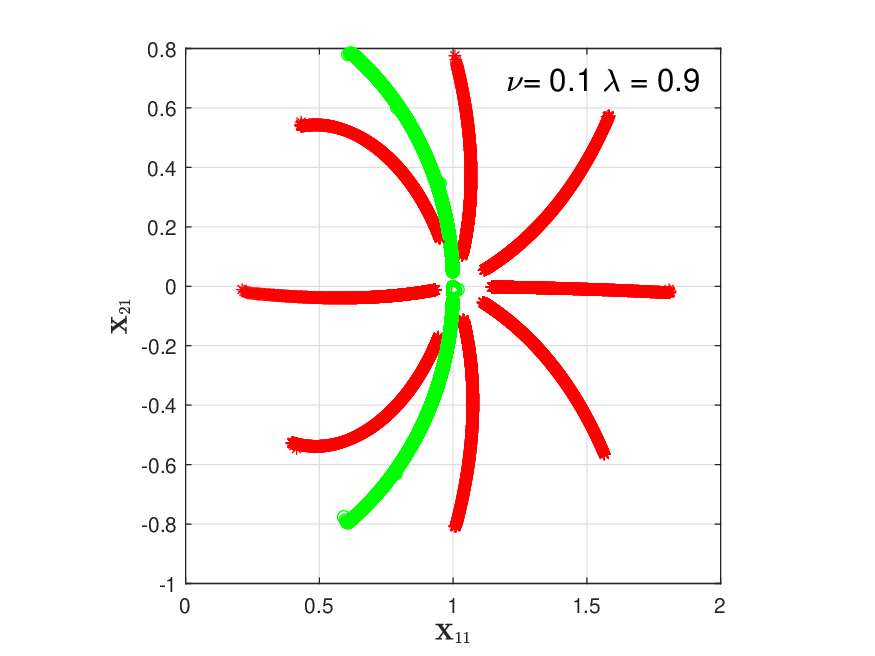}
\end{center}
\caption{A small $\nu$ and a large $\lambda$ does not seem to have an initial convergence to a subpsace.}
\label{fig:EigFig4}
\end{figure}

\section{Summary and Outlook}

We have developed and analyzed two approaches for solving optimization problems on the Stiefel manifold using damped dynamical systems: the Lagrange formulation and the projected gradient method. Both methods demonstrate asymptotic stability and convergence to local minima, confirmed by numerical experiments for eigenvalues and the Procrustes problems.

We have not discussed in detail the choice of (optimal)  damping parameters $\eta$, time steps $h$ and the matrix $\bK$. For very special problem settings that might be possible but generally it's more reasonable to aim for some kind of iterative approach choosing parameters iteratively. Moreover, experience has shown that different parameters may be used to find more interesting local minima~\cite{Ogren2025}.

While we used the symplectic Euler method, developing geometric integrators that preserve the manifold structure and energy properties could improve numerical stability and efficiency.

These methods can be applied to more general Riemannian manifolds beyond the Stiefel case that could lead to unified frameworks for constrained optimization.

\end{document}